\documentclass[12pt]{amsart}
\usepackage[top=1in, bottom=1in, left=1in, right=1in]{geometry}

\usepackage{amssymb,mathrsfs,amsmath}

\newtheorem{theorem}{Theorem}[section]
\newtheorem{proposition}{Proposition}[section]
\newtheorem{lemma}[theorem]{Lemma}
\newtheorem{corollary}[theorem]{Corollary}
\theoremstyle{remark}
\newtheorem{remark}{Remark}[section]

\newcommand{\bsp}{\begin{split}}
\newcommand{\esp}{\end{split}}
\newcommand{\be}{\begin{equation}}
\newcommand{\ee}{\end{equation}}
\newcommand{\bes}{\begin{equation*}}
\newcommand{\ees}{\end{equation*}}
\newcommand{\bv}\boldsymbol{}

\numberwithin{equation}{section}

\renewcommand{\pmod}[1]{~({\rm mod}\,#1)}

\begin{document}

\title{A new proof of Hal\'asz's Theorem, and its consequences}
\author[A. Granville]{Andrew Granville}
\address{AG: D\'epartement de math\'ematiques et de statistique\\
Universit\'e de Montr\'eal\\
CP 6128 succ.~Centre-Ville\\
Montr\'eal, QC H3C 3J7\\
Canada; 
and Department of Mathematics \\
University College London \\
Gower Street \\
London WC1E 6BT \\
England.
}
\thanks{Andrew Granville has received funding in aid of this research from the
European Research Council  grant agreement n$^{\text{o}}$ 670239, and from NSERC Canada under the CRC program. Adam Harper was supported, for parts of the research, by a postdoctoral fellowship from the Centre de recherches math\'{e}matiques in Montr\'{e}al, and by a research fellowship at Jesus College, Cambridge. Kannan Soundararajan was partially supported by NSF grant DMS 1500237, and a Simons Investigator grant from the Simons Foundation. In addition, part of this work was carried out at MSRI, Berkeley during 
the Spring semester of 2017, supported in part by NSF grant DMS 1440140}
\email{{\tt andrew@dms.umontreal.ca}}
\author{Adam J Harper}
\address{AJH: Mathematics Institute, Zeeman Building, University of Warwick, Coventry CV4 7AL, England}
\email{{\tt A.Harper@warwick.ac.uk}}
\author{K. Soundararajan}
\address{KS: Department of Mathematics, Stanford University, Stanford CA 94305, USA}
\email{{\tt  ksound@stanford.edu}}
\subjclass[2010]{11N56, 11M20}
\keywords{Multiplicative functions, Hal\'asz's theorem, Hoheisel's theorem, Linnik's theorem, Siegel zeroes}

\date{\today}

\begin{abstract} Hal\'asz's Theorem gives an upper bound for the mean value of a multiplicative function $f$.  The bound is sharp for general such $f$, and, in particular, it implies that a multiplicative function with $|f(n)|\le 1$ has either mean value $0$, or is ``close to" $n^{it}$ for some fixed $t$. The proofs in the current literature have certain features that are difficult to motivate and which are not particularly flexible.  In this article we supply a different, more flexible, proof, which indicates how one might obtain asymptotics, and can be modified to short intervals and to arithmetic progressions. We use these results to obtain new, arguably simpler, proofs that there are always primes in short intervals (Hoheisel's Theorem), and that there are always primes near to the start of an arithmetic progression (Linnik's Theorem).
\end{abstract}

\maketitle


\newcommand{\eS}{\mathcal S}
\newcommand{\eL}{\mathcal L} 
\newcommand{\Lam}{\Lambda}
 
\section{Introduction}  
\label{HalRev1} 

\noindent Throughout this paper, $f$ will denote a multiplicative function. We let 
\[
F(s) = \sum_{n=1}^{\infty} \frac{f(n)}{n^{s}}, \  \
- \frac{F'(s)}{F(s)} = \sum_{n=2}^{\infty}   \frac{\Lambda_f(n)}{n^{s}}, \ \ \text{and} \ \ \log F(s) = 
\sum_{n=2}^{\infty} \frac{\Lambda_f(n)}{n^s \log n}; 
\]
and assume that all three Dirichlet series are absolutely convergent in Re$(s)>1$.  
We will restrict attention to the class ${\mathcal C}(\kappa)$ of multiplicative functions $f$ for which
\[
 |\Lambda_f(n)|\leq \kappa \Lambda(n) \ \ \textrm{for all} \ \ n\geq 1,
 \]
  where $\kappa$ is some fixed positive constant.   If $f \in {\mathcal C}(\kappa)$ 
then it follows that $|f(n)| \le d_{\kappa}(n)$ for all $n$, where $d_{\kappa}$ denotes the $\kappa$-divisor function (the coefficient of $n^{-s}$ in 
$\zeta(s)^{\kappa}$).    The class ${\mathcal C}(\kappa)$ includes many of the most useful multiplicative functions;  one 
could weaken the assumptions further if needed (for example, using a weak Ramanujan hypothesis instead, as in \cite{Sou}), but 
we restrict ourselves to ${\mathcal C}(\kappa)$ for simplicity.  
  

 \subsection{A generalization of Hal\'asz's Theorem} We begin by generalizing the classical version of Hal\'asz's Theorem to the class ${\mathcal C}(\kappa)$.  
  
\begin{theorem}[Hal\'asz's Theorem] \label{GenHal}  Let $\kappa$ be a fixed positive number, and let $x$ be large.   
Then, uniformly for all $f\in {\mathcal C}(\kappa)$,  
$$ 
\sum_{n\le x} f(n)  \ll_{\kappa} \frac{x}{\log x} \int_{1/\log x}^1 \Big( \max_{|t| \le (\log x)^{\kappa} } \Big| \frac{F(1+\sigma+it)}{1+\sigma+it}\Big| 
  \Big) \frac{d\sigma}{\sigma} + \frac{x}{\log x}(\log \log x)^{\kappa}. 
$$ 
\end{theorem}  
 
The following corollary may be easier to work with in practice.

 \begin{corollary}
 \label{HalCor} For given  $f \in {\mathcal C}(\kappa)$, define $M=M(x)$ by 
$$
 \max_{|t| \le (\log x)^{\kappa}}  \Big| \frac{F(1+1/\log x+it)}{1+1/\log x+it}\Big| =: e^{-M} (\log x)^{\kappa}.
$$ 
Then 
 $$ 
\sum_{n\le x} f(n) \ll_{\kappa}  (1+M)e^{-M} x (\log x)^{\kappa -1} + \frac{x}{\log x}(\log \log x)^{\kappa}.
$$ 
\end{corollary} 
 
Notice $M\geq -o(1)$, as $|F(1+1/\log x+it)|\leq \zeta(1+1/\log x)^\kappa \leq (1 + \log x)^\kappa$. Therefore at worst Corollary
 \ref{HalCor} matches the trivial bound
$|\sum_{n\le x} f(n)|\leq \sum_{n\le x} d_\kappa(n) \ll_{\kappa} x(\log x)^{\kappa-1}$. This lossless quality of the analytic bound is one of the key features of Hal\'asz's Theorem.
On the other hand, from Lemma \ref{lem2.7} below it follows that $e^{-M}(\log x)^{\kappa}$ is always $\gg 1$, so that at best Corollary \ref{HalCor} 
produces a bound of $\ll x (\log \log x)^{\max (\kappa, 1)}/\log x$.

Existing proofs of Hal\'asz's Theorem \cite{GSDecay, Hal1,Hal2, Mon, MV, Ten} establish an inequality like Theorem \ref{GenHal}, and then a consequence like Corollary  \ref{HalCor} (usually with $\kappa = 1$, sometimes less sharply). These proofs are all built on the original fundamental arguments of Hal\'{a}sz, and employ ``average of averages'' type arguments, whereas our proof uses a more direct and hopefully more intuitive ``triple convolution'' application of Perron's formula, which generalises naturally to diverse other situations (such as short intervals and arithmetic progressions). The reader may also consult 
\cite{IV} where the argument is presented in the simpler context of function fields, and \cite{III} where an even simpler approach is presented in the case $\kappa =1$.  

This new proof can help us to identify the ``main term(s)'' for the mean value of $f$ up to $x$. Indeed, the following result  will be proved in \cite{II}: Fix $\epsilon>0$. For $f \in {\mathcal C}(\kappa)$, there 
  are intervals $I_j,\ 1\leq j\leq \ell$, with $\ell \ll_\epsilon 1$, each of length $1/(\log x)^{C(\epsilon)}$ such that 
\begin{equation}
 \label{NextPaper}
\sum_{n\leq x} f(n) = \sum_{j=1}^\ell \frac{1}{2\pi} \int_{t\in I_j}   F(c_0+it) \frac{x^{c_0+it}}{c_0+it} dt +O\Big( \frac x{(\log x)^{1-\epsilon} }    \Big)
\end{equation}
for $c_0=1+1/{\log x}$.
These integrals can be determined under appropriate assumptions, to provide asymptotics for various interesting examples
(see section \ref{sec: examples}, which cites asymptotics determined in \cite{II}).


Theorem \ref{GenHal} and Corollary \ref{HalCor} are established in Section 2 below.  

\subsection{Multiplicative functions in short intervals}  We are interested in  the mean value of a multiplicative function   in a short  interval $[x,x+z)$ with $z$ as small as possible.
  
\begin{theorem}[Hal\'asz's Theorem for intervals] \label{ShortHal}  Given  $f \in {\mathcal C}(\kappa)$,   real $x$ and $0<\delta<\frac 14$, we have
\begin{align*} 
\sum_{x<n\le x+x^{1-\delta}} f(n)  \ll_{\kappa}& \frac{x^{1-\delta}}{\log x} \Big( \int_{1/\log x}^1 \Big( \max_{|t| \le  \frac 12 x^\delta(\log x)^{\kappa} } \Big|  F(1+\sigma+it) \Big| 
  \Big) \frac{d\sigma}{\sigma} 
  +  (\delta \log x+ \log\log x)^{\kappa} \Big) .
  \end{align*}
\end{theorem}

A simplified version of this theorem is given in the following corollary.  

\begin{corollary}\label{ShortHal2}   Given  $f \in {\mathcal C}(\kappa)$,   real $x$ and $0<\delta<\frac 14$, we have
$$ 
\sum_{x<n\le x+x^{1-\delta}} f(n)  \ll_{\kappa} \frac{x^{1-\delta}}{\log x} \Big(   (1+M_\delta)e^{-M_\delta} (\log x)^{\kappa} + (\delta \log x +\log \log x)^{\kappa}   \Big) ,
$$ 
where
 $$
 \max_{|t| \le x^\delta(\log x)^{\kappa} } \Big|  F(1+1/\log x +it) \Big| =: e^{-M_\delta} (\log x)^{\kappa}.
$$
\end{corollary} 

Theorem \ref{ShortHal} and Corollary \ref{ShortHal2} are established in Section 3.  We remark that several aspects of Corollary \ref{ShortHal2} 
are in general best possible.  For instance, taking  $f(n) = n^{it}$ for any given $|t| \leq x^{\delta}$ then $f(n)$ will be almost constant on the interval 
$(x, x+x^{1-\delta}]$, so there will be no cancellation in the mean value. This shows that the increased range of $t$ in the definition of $M_{\delta}$, 
as well as the lack of a denominator $1+1/\log x + it$ (as compared with the definition of $M$ in Corollary \ref{HalCor}), are both necessary here.     

The term $(\delta \log x)^{\kappa}$ in Corollary \ref{ShortHal2} is also necessary, in view of the following example:\
Suppose that we are given values of $f(p^k)$ for all $p\leq x^{1-\delta}$, consistent with $f\in C(\kappa)$.   We are interested in 
extending the definition of $f$ to large primes in such a way that certain sums over the corresponding $f(n)$ are maximized.     
Since any prime $p > x^{1-\delta}$  divides at most one integer in $(x, x+x^{1-\delta}]$, one can choose $f(p)$ in such a way that 
for all such multiples the corresponding values $f(mp)$ all point in the same direction. Therefore, maximizing over all ways to extend $f\in C(\kappa)$ (given the values at primes below $x^{1-\delta}$),  
we have
\[
\max_f  \Big|  \sum_{x<n<x+x^{1-\delta}} f(n) \Big| \geq\kappa  \sum_{p>x^{1-\delta}}
\sum_{x<mp<x+x^{1-\delta}} |f(m)| \geq  \kappa \sum_{m<x^{\delta}}  |f(m)| \sum_{x/m<p<x/m+x^{1-\delta}/m} 1.
\]
Hence, using the prime number theorem, 
\[
\max_{X<x<2X} \max_f  \Big|  \sum_{x<n<x+x^{1-\delta}} f(n) \Big| \gg
 \kappa \sum_{m<X^{\delta}}  |f(m)| X^{-\delta} \sum_{X/m<p<2X/m} 1 \gg   \frac{\kappa X^{1-\delta}}{\log X} 
  \sum_{m<X^{\delta}} \frac{ |f(m)|}m . 
 \]
 If, for example, each $|f(p)|=\kappa$ for $p\leq X^{\delta}$ then this is
 $\gg_\kappa \frac{ X^{1-\delta}}{\log X}  (\delta \log X)^\kappa$.
\subsection{Slow variation of mean-values of multiplicative functions.}  We are also interested in bounding how much the mean value of a multiplicative function can vary in short {\em multiplicative} intervals; a general version of this principle forms the main input in \cite{Sou}.
The example $f(n)=n^{i\alpha}$ shows that the mean value $\frac{1}{x} \sum_{n \leq x} f(n)$ (which is $\sim x^{i\alpha}/(1+i\alpha)$) might rotate with $x$; but here we may consider the ``renormalized" mean-values 
$\frac{1}{x^{1+i\alpha}} \sum_{n\le x} f(n)$, which now varies slowly with $x$.  
%
The following result (which will be established in Section 4) establishes a general result of this flavor.   It extends the 
earlier work of Granville and Soundararajan~\cite{GSDecay} (improving on Elliott~\cite{Ell1}) obtained in the special case $\kappa=1$. 
 
\begin{theorem}[Lipschitz Theorem]\label{LipThm}  Let $f\in {\mathcal C}(\kappa)$ be given, and let $x$ be large.  Let $t_1 = t_1(x)$ 
be a point in $[-(\log x)^{\kappa}, (\log x)^{\kappa}]$ where the maximum of $|F(1+1/\log x +it)|$ (for $t$ in this interval) is attained.  For all 
$1\le w\le x^{\frac 13}$ we have 
\begin{align*}
\Big| \frac{1}{x^{1+it_1}} \sum_{n\le x} f(n) &- \frac{1}{(x/w)^{1+it_1}}  \sum_{n\le x/w} f(n) \Big| \\
&\ll \Big( \frac{\log w + (\log \log x)^2}{\log x}\Big)^{\min(\kappa(1-\frac{2}{\pi}),1)} (\log x)^{\kappa -1} \log \Big( \frac{\log x}{\log ew} \Big). 
\end{align*} 
%
%
%
If $f\in {\mathcal C}(\kappa)$ is always non-negative, then the estimate above holds with the improved exponent $\min (\kappa(1-\frac{1}{\pi}),1)$.  
%
\end{theorem}


In Examples 2 and 3 of Section 9, we show that when $\kappa =1$ there are examples of functions in ${\mathcal C}(1)$ where the exponent $1-2/\pi$ that 
appears above is attained, and there are examples where the exponent $1-1/\pi$ for non-negative functions is attained.   
In \cite{GSDecay} it was suggested that the right Lipschitz exponent for $\kappa =1$ might be as large as $1$; we now see that this is not the case.  

 
\subsection{Multiplicative functions in arithmetic progressions}
We now establish  analogues of Hal\'asz's Theorem \ref{GenHal} and Corollary \ref{HalCor} for multiplicative functions in arithmetic progressions. There is a strong analogy with the situation in short intervals, with Dirichlet characters $\chi$ modulo $q$ being introduced here and playing a similar role as the ``continuous characters'' $(n^{it})_{|t| \leq x^{\delta}}$ in the short interval theorems.

\begin{theorem} [Hal\'asz's Theorem for arithmetic progressions]  \label{GenHalq}
Given  $f \in {\mathcal C}(\kappa)$, if $x\geq q^4$ and $(a,q)=1$ then 
\begin{align*}
\sum_{\substack{n\leq x\\ n\equiv a \pmod q}}  f(n)   \ll \frac 1 {\phi(q)} \frac{x}{\log x} \Big( \int_{1/\log x}^1 \Big( \max_{\substack{ \chi \pmod q \\ |t| \le (\log x)^{\kappa} }} \Big| \frac{F_\chi(1+\sigma+it)}{1+\sigma+it}\Big| 
  \Big) \frac{d\sigma}{\sigma}   + (\log (q\log x))^{\kappa}  \Big),
\end{align*}
where $F_{\chi}(s) := \sum_{n=1}^{\infty} {f(n) \chi(n)}/{n^{s}}$ denotes the twisted Dirichlet series.
\end{theorem}

 For each character $\chi \pmod q$ we define $M_\chi= M_\chi(x)$ as follows:
$$
\max_{ |t| \le (\log x)^{\kappa} }  \Big| \frac{F_\chi(1+1/\log x+it)}{1+1/\log x+it}\Big| =: e^{-M_\chi} (\log x)^{\kappa}.
$$
The following corollary gives a simplified version of Theorem \ref{GenHalq}.

 \begin{corollary}
 \label{HalCorq} Given  $f \in {\mathcal C}(\kappa)$ let $M_q := \min_\chi M_\chi(x)$. Then
$$ 
\sum_{\substack{n\leq x\\ n\equiv a \pmod q}}  f(n)  \ll \frac 1 {\phi(q)}  \frac{x}{\log x} \left(  (1+M_q)e^{-M_q}  (\log x)^{\kappa} + (\log  (q\log x))^{\kappa}  \right) . 
$$ 
\end{corollary} 

Theorem \ref{GenHalq} and Corollary \ref{HalCorq} will be proved in Section 3.

\subsection{Isolating characters that correlate with a multiplicative function, and the pretentious large sieve.}   
Given a multiplicative function $f\in {\mathcal C}(\kappa)$ 
we shall show in Section 5 that there cannot be many characters $\chi \pmod q$ that correlate with $f$.   For instance, there cannot be many characters $\chi \pmod q$ 
for which the quantity $M_{\chi}(x)$ is small.   By extracting the few characters $\chi$ that correlate with $f$, one can obtain a more refined understanding of the 
distribution of $f$ in arithmetic progressions.  For a given $f \in {\mathcal C}(\kappa)$ and a character $\chi \pmod q$ define
 \[
 S_f(x,\chi):= \sum_{n\leq x} f(n) \overline{\chi}(n).
\]

\begin{theorem} 
\label{thm1.8}  Let $f \in {\mathcal C}(\kappa)$, let $X$ be large, and $q$ be a natural number with $q\le X^{\frac 14}$.   Put $Q =q \log X$.   For any natural number $J \ge 1$, there are at most $J$ characters $\chi \pmod q$ for which 
\begin{equation} 
\label{1.2}
\max_{\sqrt{X} \le x \le X^2} \frac{|S_f(x,\chi)|}{x} \gg_{\kappa, J} (\log X)^{\kappa-1} \Big( \frac{\log Q}{\log X} \Big)^{\kappa (1-\frac{1}{\sqrt{J+1}})} 
\log \Big( \frac{\log X}{\log Q} \Big). 
\end{equation} 
If ${\mathcal X}_J$ denotes the set of at most $J$ characters for which \eqref{1.2} holds, then for all $x$ in the range $\sqrt{X} \le x \le X^2$ and all 
reduced residue classes $a\pmod q$ we have 
$$ 
\Big| \sum_{\substack{n\le x\\ n\equiv a \pmod q}} f(n) - \sum_{\chi \in {\mathcal X}_J} \frac{\chi(a)}{\phi(q)} S_f(x,\chi) \Big|  
\ll_{\kappa, J} \frac{x}{\phi(q)} (\log x)^{\kappa -1} \Big( \frac{\log Q}{\log X} \Big)^{\kappa (1-\frac{1}{\sqrt{J+1}})} 
\log \Big( \frac{\log X}{\log Q}\Big). 
$$
\end{theorem} 

 

Theorem \ref{thm1.8} gives non-trivial bounds once $X \ge q^{A}$ for a sufficiently large constant $A$.  Precursors to this 
result may be found in the work of Elliott \cite{Ell2}, where the effect of the principal character was removed and the difference 
estimated.    Theorem \ref{thm1.8} together with related variants will be established in Section 6. 

We next describe one of these variants, which we call ``the pretentious large sieve."   For any sequence 
of complex numbers $a_n$ with $n\le x$, the orthogonality relation for characters gives 
$$
\frac{1}{\phi(q)} \sum_{\chi \pmod q} \Big| \sum_{n\le x} a_n \chi (n) \Big|^2 =\sum_{(a,q)=1} \Big| \sum_{\substack{n\le x \\ n\equiv a\pmod q}} a_n \Big|^2, 
$$ 
and, using the Cauchy--Schwarz inequality, this is 
$$ 
\le \sum_{(a,q)=1} \Big( \sum_{\substack {n\le x \\ n\equiv a \pmod q}} 1 \Big) \Big(\sum_{\substack{n\le x\\ n\equiv a \pmod q} } |a_n|^2 \Big) 
\le \Big( \frac{x}{q}+1\Big) \sum_{\substack{n\le x \\ (n,q)=1}} |a_n|^2.
$$ 
This simple large sieve inequality is, in general, the best one can do.   For a multiplicative 
function $f\in {\mathcal C}(\kappa)$ this becomes, when $x\ge q^2$,  
$$  
\sum_{\chi \pmod q} \Big| \sum_{n\le x} f(n) \overline{\chi}(n) \Big|^2 \le 
\phi(q) \sum_{(a,q)=1} \Big( \sum_{\substack{n\le x \\ n\equiv a \pmod q}} d_{\kappa}(n) \Big)^2 
\ll_{\kappa} ( x (\log x)^{\kappa -1})^2,
$$ 
using Lemma \ref{Shiu} below.
Our work shows that if a handful of characters are removed, then the above large sieve estimate can
be improved in the range $x \ge q^A$ for a suitably large $A$.   


 
\begin{theorem} [The pretentious large sieve] \label{PLSthm}
Suppose that $f \in {\mathcal C}(\kappa)$.    Let $X$, $q$ and $Q$ be as in Theorem \ref{thm1.8}.  Let $J\ge 1$ be a natural number, and let ${\mathcal X}_J$ 
consist of the at most $J$ characters $\chi$ for which \eqref{1.2} holds.  Then for all $\sqrt{X} \le x \le X^2$ we have 
\[
  \sum_{\substack{\chi\pmod q \\ \chi\not \in {\mathcal X}_J}}	|S_f(x,\chi) |^2
  \ll_{\kappa,J}   \Big( x (\log x)^{\kappa -1}\Big)^2     
\Big(   \Big( \frac{\log Q}{\log x} \Big)^{ \kappa (1 - \frac{1}{\sqrt{J+1}} )} \log\Big( \frac{\log x} {\log Q}\Big) \Big)^2 .
\]
\end{theorem}


\subsection{Primes in short intervals and arithmetic progressions.}  We now give some applications of our work to counting primes in short intervals and 
arithmetic progressions.  We shall obtain qualitative versions of Hoheisel's theorem that there are many primes in intervals $[x,x+x^{1-\delta}]$ for some $\delta >0$, 
and Linnik's theorem that the least prime $p\equiv a\pmod{q}$ is below $q^L$ for an absolute constant $L$.   In both situations, more precise 
results are known, but our work is perhaps simpler in comparison with other approaches, avoiding the use of log-free zero-density estimates 
and refined zero-free regions.   In the case of Linnik's theorem, our work bears several features in common with Elliott's proof in \cite{Ell3}.  

\begin{corollary}\label{Hoh}  Let $\delta >0$ be small, and $x$ be a large real number.  Then, with implied constants being absolute,   
\begin{equation}\label{1.3}
\sum_{x<n\leq x+x^{1-\delta}} \Lambda(n)  =x^{1-\delta} \Big( 1 + O\Big( \delta^{\frac 15} + \frac{1}{(\log x)^{\frac 1{20}} }\Big) \Big).
\end{equation}
\end{corollary}




Now we turn to primes in arithmetic progressions, where the story is more complicated because of the 
possibility of exceptional characters that might pretend to be the M{\" o}bius function.   

\begin{theorem}\label{thm1.11} Let $q$ be a natural number, and $x \ge q^{10}$ be large.   Put $Q = q\log x$ and let $J\ge 1$ 
be a natural number.   Let ${\mathcal X}_J$ denote the set of non-principal characters $\chi \pmod q$ such that 
$$ 
\max_{|t| \le (\frac{\log x}{\log Q})} \prod_{Q\le p\le x} \Big| 1- \frac{\chi(p)}{p^{1+it}}\Big| \gg_J \Big( \frac{\log x}{\log Q}\Big)^{\frac{1}{\sqrt{J+1}}}. 
$$ 
Then ${\mathcal X}_J$ has at most $J$ elements, and for any non-principal character $\psi \not \in {\mathcal X}_J$ we have 
\begin{equation} 
\label{1.4} 
\sum_{n\le x} \Lambda(n) \psi(n) \ll x \Big(\frac{\log Q}{\log x}\Big)^{1-\frac{1}{\sqrt{J+1}} -\epsilon}. 
\end{equation} 
Further for any reduced residue class $a \pmod q$ we have 
\begin{equation} 
\label{1.5} 
\sum_{\substack{n\le x\\ n\equiv a\pmod q}} \Lambda(n) - \frac{1}{\phi(q)} \sum_{\chi \in {\mathcal X}_J} \overline{\chi}(a) \sum_{n\le x} \Lambda(n) \chi(n) 
= \frac{x}{\phi(q)} +O_J\Big(\frac{x}{\phi(q)} \Big(\frac{\log Q}{\log x}\Big)^{1-\frac{1}{\sqrt{J+1}}-\epsilon}   \Big). 
\end{equation}
\end{theorem}

\begin{corollary} \label{cor1.12} Let $q$ be a natural number and $x\ge q^{10}$ be large.  
Set $Q=q\log x$. Then, for any reduced residue class $a\pmod q$, 
\begin{equation}
\label{1.6}
\sum_{\substack{ n\le x \\ n\equiv a\pmod q}} \Lambda(n) = \frac{x}{\phi(q)} + O\Big( \frac{x}{\phi (q)} \Big( \log \Big(\frac{\log x}{\log Q} \Big) \Big)^{-1} \Big), 
\end{equation} 
unless there is a non-principal quadratic character $\chi \pmod q$ such that 
\begin{equation} 
\label{1.7} 
\sum_{Q \le p \le x} \frac{1+\chi(p)}{p} \le 30 \log \log \Big(\frac{\log x}{\log Q}\Big) + O(1). 
\end{equation} 
If this exceptional character exists, then 
\begin{equation} 
\label{1.8}
\sum_{\substack{ n\le x \\ n\equiv a\pmod q}} \Lambda(n) =\frac{x}{\phi(q)} + \frac{\chi(a)}{\phi(q)} \sum_{n\le x} \Lambda(n) \chi(n) + O\Big( \frac{x}{\phi(q)} 
\Big( \frac{\log Q}{\log x}\Big)^{1-\frac{1}{\sqrt{2}}- \epsilon}\Big). 
\end{equation} 
\end{corollary} 
 
 Theorem \ref{thm1.11} also leads to an improvement of the large sieve estimate 
 for primes, once one removes the effect of a few characters.  
 
\begin{corollary}   [The pretentious large sieve for the primes] \label{cor1.13}  Let $q$, $x$ and $Q$ be as in Theorem \ref{thm1.11}, and 
let $\widetilde{\mathcal X}_J$ denote the set ${\mathcal X}_J$ of Theorem \ref{thm1.11} extended to include the principal character.  Then 
$$ 
\sum_{\chi \not \in \widetilde{\mathcal X}_J} \Big| \sum_{n\le x} \Lambda(n) \chi(n) \Big|^2 \ll x^2 \Big( \frac{\log Q}{\log x}\Big)^{2 (1-\frac{1}{\sqrt{J+1}}) - \epsilon}. 
$$ 
\end{corollary}

Theorem \ref{thm1.11} and Corollaries \ref{cor1.12} and \ref{cor1.13} will be established in Section 7.  
Finally in Section 8 we shall deduce the following qualitative form of  Linnik's theorem.  

\begin{corollary}[Linnik's Theorem] \label{Linnik} 
There exists a constant $L$ such that if $(a,q)=1$ then there is a prime
$p\equiv a \pmod q$ with $p\ll q^L$. 
\end{corollary}

\section{Hal\'asz's theorem for the class ${\mathcal C}(\kappa)$:  Proofs of Theorem \ref{GenHal} and Corollary \ref{HalCor} }
  \label{Sec2}
 
\noindent  In this section, we establish Theorem \ref{GenHal}.  The strategy of the proof also 
generalizes readily to give variants of Hal{\' a}sz's theorem in short intervals and arithmetic progressions, 
and we shall give these versions in the next section.  

 Let $p(n)$ and $P(n)$ denote (respectively) the smallest and largest prime  
factors of $n$.   Given a parameter $y \ge 10$, define the multiplicative functions 
$s(\cdot)$ (supported on {\sl small} primes) and $\ell(\cdot)$ (supported on {\sl large} primes)  by
\[
 s(p^k)=\begin{cases} f(p^k) \\  0   \end{cases}  
 \ \ \text{and} \ \ 
\ell(p^k)=\begin{cases} 0 & \text{if} \ p\leq y,\ k\geq 1 \\
       f(p^k) & \text{if} \ p> y,\ k\geq 1;
      \end{cases}
\]
so that $f$ is the convolution of $s$ and $\ell$. Therefore setting (for $s$ with Re$(s)>1$)
\[
 {\mathcal S}(s) = \sum_{n\geq 1}\frac{s(n)}{n^{s}}  \ \ \textrm{  and  } \  \
{\mathcal L}(s) 
=\sum_{n\geq 1}\frac{\ell(n)}{n^{s}}  , \ \  \text{we have} \ \ F(s) = \eS(s) \eL(s). 
\]
We define $\Lambda_s$ and $\Lambda_\ell$ analogously.  Note that $s$, $\ell$, $\eS$ and $\eL$ all 
depend on the parameter $y$.   We also remark that since ${\mathcal S}$ is defined by a finite Euler product, 
${\mathcal S}(s)$ converges and is analytic for Re$(s)>0$.  
The following proposition, and variants of it, will play a key role in our 
proof of Hal{\' a}sz's theorem.

\begin{proposition}
\label{prop2.1} 
\label{xkeyidr1} Let $10 \le y\le \sqrt{x}$, let $\eta =1/{\log y}$ so that $0<\eta <1/2$. Set $c_0=1+ 1/{\log x}$. 
Then, for any $1\le T \le x^{9/10}$,  
\begin{equation}
\label{xkeyidr2}
\int_0^{\eta}\int_0^{\eta} \frac{1}{\pi i } \int_{c_0-iT}^{c_0+iT} 
\eS(s-\alpha-\beta)\eL(s+\beta)
 \sum_{\substack{ y<m<x/y \\ y< n <x/y}} \frac{\Lambda_\ell(m)}{m^{s-\beta}} \frac{\Lambda_\ell(n)}{n^{s+\beta}}
 \frac{x^{s-\alpha-\beta}}{s-\alpha-\beta}\ ds\ d\beta\ d\alpha
\end{equation} 
equals 
\begin{equation} 
\label{xkeyidr3} 
\sum_{n\le x} f(n)  +O\left(   \frac{x}{\log x} (\log y)^{\kappa}  + \frac{x(\log x)^{\kappa}}{T} \right) .
\end{equation}
 \end{proposition}
 

 The crucial point here is the presence of multiple Dirichlet polynomials inside the integral (visibly four, but really three since the factors $\mathcal{S}$ and $\mathcal{L}$ go together). To understand why this is useful one may think of the  circle method for ternary problems (for example, sums of three primes).  In such situations, three generating functions are 
 involved, and a satisfactory bound is obtained by using an $L^{\infty}$ estimate for one of the generating functions and then using $L^2$ estimates (Parseval's identity) to bound the 
 remaining two.  Proposition \ref{prop2.1} allows us to employ a similar strategy here, bounding ${\mathcal S}(s-\alpha-\beta){\mathcal L}(s+\beta)$ in magnitude, and then using a Plancherel bound (see Lemma \ref{MeanSquarePrimes}) for the remaining two sums.  
 


\subsection{Introducing many convolutions} 

The kernel of Proposition \ref{prop2.1} is contained in the following slightly simpler identity, which we will tailor (see Lemma \ref{keyidr1}) to obtain a more flexible variant.  

\begin{lemma} \label{keyid} For any $x >2$ with $x$ not an integer, and any $c>1$, we have
$$
\sum_{2\le n\le x} \Big( f(n)- \frac{\Lambda_f(n)}{\log n} \Big) =
\int_0^{\infty} \int_0^{\infty} \frac{1}{2\pi i} \int_{c-i\infty}^{c+i\infty} \frac{F^{\prime}}{F}(s+\alpha) F^{\prime}(s+\alpha+\beta)
\frac{x^s}{s} ds \ d\beta \ d\alpha.
$$
\end{lemma}
\begin{proof}  We give two proofs.  In the first proof, we interchange the integrals over $\alpha$ and $\beta$, and $s$.
First perform the integral over $\beta$.  Since $\int_{\beta=0}^{\infty} F^{\prime}(s+\alpha+\beta) d\beta = 1-F(s+\alpha)$
we are left with
$$
\frac{1}{2\pi i}   \int_{c-i\infty}^{c+i\infty} \int_0^{\infty} \Big(\frac{F^{\prime}}{F}(s+\alpha) - F^{\prime}(s+\alpha) \Big)
\frac{x^s}{s} \ d\alpha \ ds,
$$
and now performing the integral over $\alpha$ this is
$$
\frac{1}{2\pi i} \int_{c-i\infty}^{c+i\infty} \Big(- \log F(s) + F(s) -1 \Big) \frac{x^s}{s}ds.
$$
Perron's formula now shows that the above matches the left hand side of the stated identity.

For the second (of course closely related) proof, Perron's formula gives that
$$
\frac{1}{2\pi i} \int_{c-i\infty}^{c+i\infty} \frac{F^{\prime}}{F}(s+\alpha) F^{\prime}(s+\alpha+\beta) \frac{x^s}{s}ds
= \sum_{\substack{{ 2\le \ell, m } \\ {\ell m \le x}}}\frac{ \Lambda_f(\ell)}{\ell^{\alpha}} \frac{f(m) \log m}{m^{\alpha+\beta}}.
$$
Integrating now over $\alpha$ and $\beta$ gives
$$
\sum_{\substack{{ 2\le \ell, m } \\ {\ell m \le x}} }\frac{\Lambda_f(\ell) f(m)}{\log (\ell m)} = \sum_{2\le n\le x} \Big(f(n) - \frac{\Lambda_f(n)}{\log n}\Big),
$$
where we use that $\sum_{\ell m =n } \Lambda_f(\ell) f(m) = f(n)\log n$ (which follows upon 
comparing coefficients in $(-F'/F)(s)\cdot F(s)=-F'(s)$).
\end{proof}


\begin{lemma} 
\label{keyidr1} Let $F$, $\eS$, $\eL$, $x$ and $y$ be as above.  Let $\eta >0$ and $c>1$ be real numbers. 
Then 
\begin{equation}
\label{keyidr2}
\int_{\alpha=0}^{\eta}\int_{\beta=0}^{2\eta} \frac{1}{2\pi i } \int_{c-i\infty}^{c+i\infty} 
\eS(s)\eL(s+\alpha+\beta) \frac{\eL^{\prime}}{\eL}(s+\alpha)\frac{\eL^{\prime}}{\eL}(s+\alpha+\beta) \frac{x^s}{s} ds d\beta d\alpha
\end{equation} 
\begin{equation} 
\label{keyidr3} 
= \sum_{n\le x} f(n)  -\sum_{mn\le x} s(m) \frac{\ell(n)}{n^{\eta}} 
- \int_0^{\eta} \sum_{mkn \le x} s(m) \frac{\Lambda_\ell(k)}{k^{\alpha}} \frac{\ell(n)}{n^{2\eta+\alpha}} d\alpha.
\end{equation}
\end{lemma}
\begin{proof}  
First we perform the integral over $\beta$ in \eqref{keyidr2}, obtaining 
$$ 
\int_{\alpha=0}^{\eta} \frac{1}{2\pi i} \int_{c-i\infty}^{c+i\infty}  \eS(s) \frac{\eL^{\prime}}{\eL}(s+\alpha) \Big( 
\eL(s+\alpha+2\eta ) - \eL(s+\alpha) \Big) \frac{x^s}{s} ds d\alpha. 
$$ 
The term arising from $\eL(s+\alpha+2\eta)$ above gives the third term in \eqref{keyidr3}, using Perron's formula.
The term arising from $\eL(s+\alpha)$ gives
$$ 
-\int_0^{\eta} \frac{1}{2\pi i} \int_{c-i\infty}^{c+i\infty} \eS(s) \eL^{\prime}(s+\alpha) \frac{x^s}{s} ds \ d\alpha 
= \frac{1}{2\pi i} \int_{c-i\infty}^{c+i\infty} \eS(s) (\eL(s) -\eL(s+\eta)) \frac{x^s}{s} ds, 
$$ 
upon evaluating the integral over $\alpha$.  Perron's formula now matches the two terms 
above with the first two terms in \eqref{keyidr3}.  This establishes our identity. 
\end{proof} 

 To bound the contributions of the second and third terms in \eqref{keyidr3}, as well as other related quantities, we  use a well known result of Shiu establishing a Brun--Titchmarsh inequality for non-negative multiplicative functions. 
 
 \begin{lemma} 
 \label{Shiu}  Let $\kappa$ and $\epsilon$ be fixed positive real numbers.  Let $x$ be large, and suppose $x^{\epsilon} \le z \le x$ and 
 $q\le z^{1-\epsilon}$.   Then uniformly for all $f \in {\mathcal C}(\kappa)$ we have 
 $$ 
 \sum_{\substack{ x-z \le n \le x \\ n\equiv a \pmod q }} |f(n) | \ll_{\kappa,\epsilon} \frac{z}{\phi(q)} \frac{1}{\log x} \exp\Big( \sum_{\substack{ p\le x \\ p\nmid q}} \frac{|f(p)|}{p}\Big),  
 $$ 
 where $a\pmod q$ is any reduced residue class. 
 \end{lemma} 
 \begin{proof}  This follows from Theorem 1 of Shiu \cite{Shiu}. 
 \end{proof} 
 

\begin{lemma}  \label{errors23}  Keep notations as above.   For $\eta =1/\log y$, we have 
$$ 
\sum_{mn \le x} |s(m)| \frac{|\ell(n)|}{n^{\eta}} + \int_0^{\eta} \sum_{mkn \le x} |s(m)|\frac{|\Lambda_\ell(k)|}{k^{\alpha}}  \frac{|\ell(n)|}{n^{2\eta+\alpha}} d\alpha  \ll_{\kappa}   \frac{x}{\log x} (\log y)^{\kappa}. 
$$ 
\end{lemma} 
\begin{proof}   Applying Lemma \ref{Shiu} (with $z=x$ and $q=1$ there) we find that 
$$ 
\sum_{mn \le x} |s(m)| \frac{|\ell(n)|}{n^{\eta}} \ll_{\kappa} \frac{x}{\log x} \exp\Big( \sum_{p\le y} \frac{|f(p)|}{p} 
+ \sum_{y< p\le x} \frac{|f(p)|}{p^{1+\eta}} \Big) \ll_{\kappa} 
\frac{x}{\log x} (\log y)^{\kappa}. 
$$  
As $\sum_{k\le K} {|\Lambda(k)|}/{k^{\alpha}} \ll K^{1-\alpha}$ by the Chebyshev bound $\psi(x)\ll x$, we have 
$$ 
 \int_0^{\eta} \sum_{mkn \le x} |s(m)| \frac{| \Lambda(k)|}{k^{\alpha}}  \frac{|\ell(n)|}{n^{2\eta+\alpha}}d\alpha  
\ll_{\kappa} \int_0^{\eta}  x^{1-\alpha} \sum_{mn \le x}\frac{ |s(m)|}{m^{1-\alpha}}  \frac{|\ell(n)|}{n^{1+2\eta}} d\alpha ,
$$ 
which is 
$$ 
\ll_{\kappa} \int_0^{\eta} x^{1-\alpha} \prod_{p \leq y} \Big( 1- \frac{1}{p^{1-\alpha}}\Big)^{-\kappa} \prod_{y < p\le x} 
\Big(1-\frac{1}{p^{1+2\eta}}\Big)^{-\kappa} d\alpha  
\ll_{\kappa}  \frac{x}{\log x} (\log y)^{\kappa},
$$ 
as desired.
\end{proof}


\subsection{Tailoring the integral} \label{sec:Tailor}
Now we turn to the task of bounding the integrals in \eqref{keyidr2}, with a view to proving Proposition \ref{xkeyidr1} and ultimately Hal\'asz's Theorem \ref{GenHal}.  Fix $\alpha, \beta\in [0,2\eta]$. If we write the terms in the Dirichlet series in \eqref{keyidr2} as $a,b,c,d$ respectively, then by Perron's formula the integral over $s$ equals a sum of the shape $\sum_{abcd\leq x} (bcd)^{-\alpha} (bd)^{-\beta}$. Both $c$ and $d$ are $>y$, by the definition of $\eL$, and so $a$, $b$, $c$, and $d$ must all be $< x/y$. Hence in \eqref{keyidr2} we may replace
\[
 \frac{\eL^{\prime}}{\eL}(s+\alpha) = -\sum_{p(n)>y} \frac{\Lambda_\ell(n)}{n^{s+\alpha}}  \ \ \ \text{by} \ \ \ 
- \sum_{y<n<x/y} \frac{\Lambda_\ell(n)}{n^{s+\alpha}}  ,
 \]
and similarly for $(\eL^{\prime}/{\eL})(s+\alpha+\beta)$. Moreover, with these Dirichlet series truncated to finite sums we may move the line of integration to Re$(s)=c_{\alpha,\beta}$ for any $c_{\alpha,\beta}>0$.

Choose $c_{\alpha,\beta}=c_0-\alpha-\beta/2$, so that (after a change of variables) the inner integral over $s$ in \eqref{keyidr2} is 
\begin{align}\label{integral} 
\frac{1}{2\pi i} \int_{c_0-i\infty}^{c_0+i\infty} \eS(s-\alpha-\beta/2) \Big(\sum_{y< m < x/y} \frac{\Lambda_\ell(m)}{m^{s-\beta/2}}\Big) 
\Big( \sum_{y< n< x/y} \frac{\Lambda_\ell(n)}{n^{s+\beta/2}}\Big) \eL(s+\beta/2)  
\frac{x^{s-\alpha-\beta/2}}{s-\alpha-\beta/2} ds.
\end{align} 

We now show that the integral in \eqref{integral} may be truncated at $T$ with an error term that is at most the second error term in Proposition \ref{xkeyidr1}.
To do this, we shall find useful the following consequence of the Chebyshev bound $\psi(x) \ll x$: for all $0< |\lambda| \le 1$ 
\begin{equation}\label{sumprimebeta} 
\sum_{y< m < x/y} \frac{\Lambda(m)}{m^{1-\lambda}} \ll \Big(\frac{x}{y}\Big)^{\max(\lambda, 0)} \min\Big( \frac{1}{|\lambda|}, \log x\Big). 
\end{equation}

\begin{lemma}
\label{keyid3}  For $x^{9/10} \ge T\ge 1$, the part of the integral  in \eqref{integral} with $|t|\geq T$, integrated over 
$0\leq \alpha,\ \beta\leq 2/{\log y}$, is bounded by
$O(x(\log x)^{\kappa}/T)$. 
\end{lemma}
\begin{proof}   
The error introduced in truncating the integral in \eqref{integral} at $T$ is, by the quantitative Perron formula (see Chapter 17  of Davenport \cite{Dav}), 
$$
\ll_{\kappa} \sum_{m, n}  \sum_{y \le k_1, k_2 \le x/y}  \frac{ \Lambda(k_1)  \Lambda(k_2)}{k_1^{\alpha} k_2^{\alpha+\beta}}  |s(m)|
\frac{|\ell(n)|}{n^{\alpha+\beta}} \Big(\frac{x}{k_1 k_2 mn}\Big)^{c_0 -\alpha-\beta/2} \min \Big(1, \frac{1}{T|\log (x/k_1k_2 mn)|}\Big).
$$
 We bound this error term by splitting the values of $k_1 k_2  m n= N$ into various ranges.
The terms with $N \le x/2$ or $N>3x/2$ are bounded by
 \begin{align*}
& \ll \frac{x^{1-\alpha-\beta/2}}{T} \sum_{y\le k_, k_2 \le x/y} \frac{\Lambda(k_1)\Lambda(k_2)} {k_1^{1-\beta/2}k_2^{1+\beta/2}}
 \sum_{m} \frac{|s(m)|}{m^{1-\alpha-\beta/2}} \sum_{n} \frac{|\ell(n)|}{n^{c_0+\beta/2}} 
\\
&\ll \frac{x^{1-\alpha}}{T} \Big( \min\Big( \frac{1}{\beta},\log x \Big)\Big)^2   (\log y)^{\kappa}
\end{align*}
by \eqref{sumprimebeta}.
Integrating over $\alpha$ and $\beta$ leads to a bound $\ll x(\log y)^{\kappa}/T\ll x(\log x)^{\kappa}/T$.

Now we turn to the terms with $k_1k_2mn=:N$ where $x/2 \le N \le 3x/2$. Note that $m$ is the largest divisor of $N$ free of prime factors $>y$, 
so that $(m,N/m)=1$.  Integrating over $\alpha$ and $\beta$, we find that
\begin{align*}
& \sum_{\substack{k_1,k_2,m,n \\ k_1k_2mn=N}} \int_0^\eta \int_0^{2\eta} \frac{ \Lambda(k_1)  \Lambda(k_2)}{k_1^{\alpha} k_2^{\alpha+\beta}}  
|s(m)|\frac{|\ell(n)|}{n^{\alpha+\beta}} \Big(\frac{x}{k_1 k_2 mn}\Big)^{c_0 -\alpha-\beta/2}  d\beta d\alpha \\
& \ll \sum_{\substack{k_1,k_2,m,n \\ k_1k_2mn=N}} \int_0^\eta \frac{ \Lambda(k_1)  }{k_1^{\alpha} }  
|s(m)| \int_0^{2\eta} \frac{\Lambda(k_2) |\ell(n)|}{(k_2 n)^{\alpha+\beta}} d\beta d\alpha \\
& \ll d_\kappa(m)\sum_{\substack{k_1, r  \\ k_1 r=N/m}} \frac{\Lambda(k_1)}{\log (N/m)}  
 \sum_{ k_2n=r} \frac{ \Lambda(k_2)d_\kappa(n)}{\log r}  ,
\end{align*}
where $r$ only has prime factors $>y$.   Using $\sum_{ab=c} \Lambda(a)d_{\kappa}(b) =\frac{1}{\kappa}  d_{\kappa}(c)\log c$ twice, we deduce that the above is 
\[
\ll  d_\kappa(m) \sum_{ k_1r=N/m} \frac{\Lambda(k_1)d_\kappa(r)}{\log (N/m)}  \ll d_\kappa(m)d_\kappa(N/m)= d_\kappa(N) .
\] Hence these terms contribute
$$
\ll \sum_{x/2 \le N \le 3x/2} d_{\kappa}(N) \min \Big(1, \frac{1}{T|\log (x/N)|}\Big) \ll \frac xT (\log x)^{\kappa-1} \log T,
$$
splitting the sum over $N$ into intervals $(x(1+\frac jT),x(1+\frac {j+1}T)]$ for $-T/2\leq j< T$,
proving the lemma. The final inequality here requires an upper bound for the sum of $d_{\kappa}(N)$ in intervals of length $x/T$ around $x$, which follows from Lemma \ref{Shiu}. 
\end{proof}

\begin{proof} [Proof of Proposition \ref{xkeyidr1}] We begin with the identity in Lemma \ref{keyidr1}, replacing the inner integral in \eqref{keyidr2} by \eqref{integral}. We truncate the inner integral obtaining the error in Lemma \ref{keyid3}. The terms in  \eqref{keyidr3} are bounded by Lemma 
\ref{errors23}. Finally we replace $\beta$ by $2\beta$.
\end{proof}

\subsection{Proof of Theorem \ref{GenHal}} \label{PfHalasz}

With Proposition \ref{prop2.1} in hand, we may proceed to the proof of Hal{\' a}sz's theorem \ref{GenHal}.  
We will bound the integral in Proposition \ref{prop2.1}  by a 
suitable application of Cauchy--Schwarz and the following simple mean value theorem for Dirichlet polynomials.

\begin{lemma} \label{MeanSquarePrimes}  For any complex numbers $\{ a(n)\}_{n\geq 1}$ and any $T\geq 1$ we have
$$
\int_{-T}^{T} \Big| \sum_{T^2 \le n\le x} \frac{a(n)\Lambda(n)}{n^{it}} \Big|^2 dt
\ll \sum_{T^2 \le n \le x} n|a(n)|^2 \Lambda(n).
$$
\end{lemma}
\begin{proof} Let $\Phi$ be an even non-negative function with $\Phi(t) \ge 1$ for
$-1\le t\le 1$, for which ${\widehat \Phi}$, the Fourier transform of $\Phi$, is compactly supported.
For example, take $\Phi(x) = \frac{1}{(\sin 1)^2} (\frac{\sin x}{x})^2$ and note that
${\hat \Phi}(x)$ is supported in $[-1,1]$.     We may then bound our integral by
\begin{align*}
&\leq \int_{-\infty}^{\infty} \Big| \sum_{T^2< n\le x} \frac{a(n)\Lambda(n)}{n^{it}} \Big|^2 \Phi\Big(\frac{t}
{T}\Big) dt
\\
&\leq   \sum_{T^2 < m,n \le x} \Lambda(m) \Lambda(n) \  |a(m)a(n)| \ 
|T\ {\widehat \Phi}(T\log (n/m))|.
\end{align*}
Since $2|a(m)a(n)| \le  |a(m)|^2+|a(n)|^2$ and by symmetry, the above is
$$
\leq \sum_{T^2 < m \le x} |a(m)|^2 \Lambda(m) \sum_{T^2\le n\le x} \Lambda(n)
|T {\widehat \Phi}(T\log (n/m))|.
$$
Since ${\widehat \Phi}$ is compactly supported we know that $|{\widehat \Phi}(t)|\ll 1$ for all $t$, and that for a given $m$,  our sum over $n$ is only supported on those values of $n$ for which $|n-m| \ll m/T$.  Therefore, using the Brun--Titchmarsh
theorem, the sum over $n$ is seen to be $\ll m$.  The lemma follows.
 \end{proof}

\begin{proof}[Proof of Theorem \ref{GenHal}] Let $x$ be large, and take $T= (\log x)^{\kappa +1}$ and $y=T^2$ in our 
work above.  The error terms in Proposition \ref{prop2.1} are $\ll x (\log \log x)^{\kappa}/\log x$, and it 
remains to handle the integral in \eqref{xkeyidr2}.  For given $\alpha$ and $\beta$ in $[0,\eta]$, the integral there is 
(since $\eta=1/\log y$ is small)  
\begin{align} 
\label{sbound}
\ll x^{1-\alpha-\beta} \Big( \max_{|t|\le T} &\frac{ |\eS(c_0-\alpha-\beta+it)\eL(c_0+\beta+it)|}{|c_0+\beta+it|} \Big)  \nonumber 
\\ &\times 
 \int_{-T}^{T} \Big| \sum_{y< m< x/y} \frac{\Lambda_\ell(m)}{m^{c_0-\beta+it}} \Big| 
\Big| \sum_{y< n< x/y} \frac{\Lambda_\ell(n)}{n^{c_0+\beta+it}} \Big| dt .
\end{align} 
Since $\alpha$, $\beta$, $c_0-1$ are all at most $1/\log y$, 
\[
\left|   \frac{\eS(c_0-\alpha-\beta+it)}{\eS(c_0+\beta+it)}   \right|\ll \exp\Big( \kappa \sum_{p\le y} \Big(\frac{1}{p^{c_0-\alpha-\beta}}- 
\frac{1}{p^{c_0+\beta}}\Big) \Big) \ll_{\kappa} 1,
\]
and so 
 \[
\frac{ |\eS(c_0-\alpha-\beta+it)\eL(c_0+\beta+it)|}{|c_0+\beta+it|} \ll 
 \Big|\frac{F(c_0+\beta+it)}{c_0+\beta+it}\Big| .
\]
Using Cauchy--Schwarz, Lemma \ref{MeanSquarePrimes} (noting that $y=T^2$), and then \eqref{sumprimebeta}, the integral in \eqref{sbound} is
\begin{equation} 
\label{HalRev9}
\ll_{\kappa} \Big( \sum_{y< m< x/y} \frac{\Lambda(m)}{m^{1-2\beta}}\Big)^{\frac 12} \Big( \sum_{y< n< x/y} 
\frac{\Lambda(n)}{n^{1+2\beta}} \Big)^{\frac 12} \ll (x/y)^{\beta} \min \Big( \log x, \frac{1}{\beta} \Big) .
\end{equation}
 Combining the last two estimates, we find that  the quantity in \eqref{sbound} is 
$$ 
\ll x^{1-\alpha } \min \Big(\log x, \frac{1}{\beta}\Big) \ 
  \max_{|t|\le T} \Big| \frac{F(c_0+\beta+it)}{c_0+\beta+it}\Big|  . 
$$ 
Writing $\sigma=\beta+1/\log x$ so that $\min \Big(\log x, \frac{1}{\beta}\Big) \asymp \frac{1}{\sigma}$, the integral  in \eqref{xkeyidr2} is 
\begin{align} \label{NB0}
&\ll \int_{\alpha=0}^1 x^{1-\alpha}  \int_{1/\log x}^{2/\log y} \Big( \max_{|t| \le T } \Big| \frac{F(1+\sigma+it)}{1+\sigma+it}\Big| 
  \Big) \frac{d\sigma}{\sigma} d\alpha \nonumber\\
  & \ll 
  \frac{x}{\log x}\int_{1/\log x}^{2/\log y} \Big( \max_{|t| \le T } \Big| \frac{F(1+\sigma+it)}{1+\sigma+it}\Big| 
  \Big) \frac{d\sigma}{\sigma}.
\end{align}
Since  $|F(1+\sigma+it)| \leq \zeta(1+\sigma)^{\kappa} \ll (1/\sigma)^\kappa$, 
$$ 
\max_{|t|\le T} \Big| \frac{F(1+\sigma+it)}{1+\sigma+it} \Big|  \le \max_{|t|\le (\log x)^{\kappa}} \Big| \frac{F(1+\sigma+it)}{1+\sigma+it}\Big| + O\Big(\frac{1}{(\sigma\log x)^{\kappa}}\Big), 
$$ 
and so the quantity in \eqref{NB0} may be bounded by 
$$ 
\ll \frac{x}{\log x} \int_{1/\log x}^{2/\log y} \Big( \max_{|t| \le (\log x)^{\kappa} } \Big| \frac{F(1+\sigma+it)}{1+\sigma+it}\Big| 
  \Big) \frac{d\sigma}{\sigma} + \frac{x}{\log x},  
$$ 
completing the proof of Theorem \ref{GenHal}.
\end{proof}

The following simple lemma establishes limits on the quality of bounds that can be deduced 
from Hal{\' a}sz's theorem, and will be useful in deducing Corollary \ref{HalCor}. 

\begin{lemma} \label{lem2.7}  Let $f \in {\mathcal C}(\kappa)$.   Suppose $0 < \sigma \le 1$, and $T$ is 
such that $T\ge 4\zeta(1+\sigma)^{\kappa/2}$.  Then 
$$
\max_{|t| \le T} \frac{|F(1+\sigma+it)|}{|1+\sigma+it|} \gg 1.
$$
\end{lemma} 
\begin{proof}   For any $c>0$, a simple contour shift argument gives 
$$ 
\frac{1}{2\pi i} \int_{c-i\infty}^{c+i\infty} \frac{2y^s}{s(s+1)(s+2)} ds   = \begin{cases} 
(1-1/y)^2 &\text{ if  } y>1 \\
0&\text{ if } 0 < y\le 1.  
\end{cases}
$$ 
Therefore
$$ 
\frac{1}{\pi i} \int_{1+\sigma -i\infty}^{1+\sigma+i\infty} \frac{F(s)}{s} \frac{2^{s}}{(s+1)(s+2)} ds = \sum_{n\le 2} f(n) (1-n/2)^2 = \frac 14.  
$$
The portion of the integral with $|t| >T$ contributes an amount that is bounded in magnitude by 
$$ 
\frac{1}{\pi} \int_{|t| >T} 2^{1+\sigma} \zeta(1+\sigma)^{\kappa} \frac{dt}{|1+it|^3} \le \frac{4}{\pi T^2} \zeta(1+\sigma)^{\kappa} \le \frac{1}{10},  
$$
and the lemma follows. 
\end{proof}

\begin{proof} [Proof of Corollary \ref{HalCor}] We use the maximum modulus principle in the region 
\[ \{ u+iv: \ 1+1/\log x \le u \le 2, \ |v|\le (\log x)^{\kappa}\} \]
to bound $\max_{|t| \le (\log x)^{\kappa} } |F(1+\sigma+it)/(1+\sigma+it)|$.  By the definition of $M=M(x)$, the maximum of 
this quantity along the vertical line segment $u=1+1/\log x$ is $e^{-M} (\log x)^{\kappa}$, and by Lemma \ref{lem2.7} this is $\gg 1$ .   On the horizontal segments $|v|= (\log x)^{\kappa}$, 
$|F(u+iv)/(u+iv)|$ may be bounded by $\le \zeta(1+1/\log x)^{\kappa} /(\log x)^{\kappa} = O(1)$, and the same bound holds on the 
vertical line segment $u=2$.  Thus 
$$ 
\max_{|t| \le (\log x)^{\kappa}  } \Big| \frac{F(1+\sigma+it)}{1+\sigma+it} \Big| \le e^{-M} (\log x)^{\kappa} + O_{\kappa}(1) \ll 
e^{-M}(\log x)^{\kappa}.
$$ 
Moreover, bounding $|F(1+\sigma+it)|$ by $\zeta(1+\sigma)^{\kappa}$ we also have 
$$ 
\max_{|t| \le (\log x)^{\kappa}  } \Big| \frac{F(1+\sigma+it)}{1+\sigma+it} \Big| \le \Big( \frac{1}{\sigma} + O(1) \Big)^{\kappa}. 
$$

Inserting these bounds into Hal{\' a}sz's Theorem \ref{GenHal} the integral there is
\[
\ll_{\kappa}  \int_{1/\log x}^{e^{M/\kappa}/\log x}   e^{-M} (\log x)^{\kappa} \frac{d\sigma}{\sigma} +
 \int_{e^{M/\kappa}/\log x}^1    \frac{d\sigma}{\sigma^{\kappa+1}}  \ll  e^{-M} (\log x)^{\kappa} \frac M\kappa+  e^{-M} (\log x)^{\kappa},
\]
and  the  corollary follows. 
\end{proof}

In some applications it may happen that even though the function $f$ belongs {\sl a priori} to some class ${\mathcal C}(\kappa)$ 
the average size of $|f(p)|$ might be smaller.  For example, if $f(n)$ is the indicator function of sums of two squares then 
$f$ lies in ${\mathcal C}(1)$, whereas the average size of $f(p)$ is $\frac 12$.  Another example comes from the normalized 
Fourier coefficients of  holomorphic eigenforms.  These lie naturally in ${\mathcal C}(2)$, but the average size of these 
coefficients at primes is smaller than $1$ in absolute value.  In such situations, one would like refined versions of 
our results taking the average behavior of $|f(p)|$ into account.   We give a sample result of this sort in the context of Corollary 
\ref{HalCor}, and it would be desirable to flesh out similar variants for our other results.  

\begin{corollary}  Let $f\in {\mathcal C}(\kappa)$ and put $G(s) = \sum_{n=1}^{\infty} |f(n)| n^{-s}$.  Suppose that 
$0 < \lambda \le \kappa$ is such that $G(1+\sigma) \ll 1/\sigma^{\lambda}$ for all $\sigma >0$.   Then 
$$ 
\sum_{n\le x} f(n) \ll_{\kappa, \lambda} (1+ M ) e^{-M } x (\log x)^{\lambda-1} + \frac{x}{\log x} (\log \log x)^{\kappa}, 
$$ 
where 
$$ 
\max_{|t| \le (\log x)^{\kappa}} \Big| \frac{F(1+1/\log x + it)}{1+1/\log x +it} \Big| =: e^{-M} (\log x)^{\lambda}. 
$$ 
\end{corollary} 
\begin{proof} We proceed as in the proof of Corollary \ref{HalCor} but now note  that $|F(1+\sigma+it)| \le |G(1+\sigma)| \ll (1/\sigma)^{\lambda}$,
and so
$$
\max_{|t| \le (\log x)^{\kappa}} \Big| \frac{F(1+\sigma+it)}{1+\sigma+it} \Big| \ll_{\kappa}
\min \Big( e^{-M} (\log x)^{\lambda}, \Big(\frac{1}{\sigma}\Big)^{\lambda}\Big).
$$
Inserting this into Hal{\' a}sz's Theorem \ref{GenHal} we obtain this corollary. 
\end{proof}


\section{Hal{\' a}sz's theorem in short intervals and progressions:  Proofs of Theorems \ref{ShortHal} and \ref{GenHalq} } 

\subsection{Short intervals} 
\begin{proof}[Proof of Theorem \ref{ShortHal}]   Since $|f(n)| \le d_{\kappa}(n)$ and 
$$
\sum_{x < n\le x+x^{1-\delta}} d_{\kappa}(n) \ll_{\kappa}  x^{1-\delta} (\log x)^{\kappa -1},
$$ 
 we 
may assume that $\delta$ is small in the theorem (else the term $(\delta \log x)^{\kappa}$ in the stated bound dominates).     We may also assume that 
$\delta>1/\log x$ else the result follows from Theorem \ref{GenHal}.   We follow the strategy of the proof of Theorem \ref{GenHal} 
with suitable modifications.   We choose $T= x^{\delta} (\log x)^{\kappa+2}$, and $y=T^2$, and define $s$, $\ell$, ${\eS}$ and ${\eL}$ as 
in Section 2.  We apply Lemma \ref{keyidr1} twice, at $x+x^{1-\delta}$ and at $x$, and subtract.  
Deferring the treatment of the secondary terms in \eqref{keyidr3} for a moment, we first explain how to bound the difference of contour integrals.

The difference of contour integrals mentioned above is (in place of \eqref{keyidr2}) 
$$ 
\int_{\alpha=0}^{\eta} \int_{\beta=0}^{2\eta} 
\frac{1}{2\pi i} \int_{c-i\infty}^{c+i\infty} {\mathcal S}(s) {\mathcal L}(s+\alpha+\beta) \frac{\mathcal L^{\prime}}{\mathcal L}(s+\alpha) 
\frac{\mathcal L^{\prime}}{\mathcal L} (s+\alpha+\beta) \frac{(x+x^{1-\delta})^s-x^s}{s} ds d\beta d\alpha. 
$$ 
As described in section \ref{sec:Tailor}, we truncate this integral at height $T$, the error term in Lemma \ref{keyid3} being acceptable since $T=x^\delta(\log x)^{\kappa+2}$.
 Then we proceed as in the proof of Hal\'asz's Theorem \ref{GenHal} in Section \ref{PfHalasz},
  with the only difference being that our integral here (in the step analogous to \eqref{sbound}) is multiplied through by a factor 
\[ | (1+x^{-\delta})^{s-\alpha-\beta} -1|\ll \min\{ |s| x^{-\delta},1\}   . \] 
Thus, in place of \eqref{NB0}, we have the bound  
\begin{equation} \label{NB01}
\ll \frac{x^{1-\delta}}{\log x} \int_{1/\log x}^{2/\log y} \Big( \max_{|t| \le T } | F(1+\sigma+it)|  \min\Big\{ 1, \frac{x^{\delta}}{|t| }\Big\}
  \Big) \frac{d\sigma}{\sigma}.
\end{equation}
The contribution of those $\sigma$ for which the  maximum of  $|F(1+\sigma+it)|$ occurs with $T\geq |t|> \frac 12 x^\delta(\log x)^\kappa$, is (since $|F(1+\sigma+it)| \ll 1/\sigma^{\kappa}$) 
\[
\ll \frac{x^{1-\delta}}{\log x} \int_{1/\log x}^{2/\log y} \frac{(1/\sigma)^{\kappa}}{(\log x)^\kappa} 
   \frac{d\sigma}{\sigma}  \ll \frac{x^{1-\delta}}{\log x}.
   \]
      This is acceptable for the estimate stated in Theorem \ref{ShortHal}.

It remains to deal with the terms corresponding to \eqref{keyidr3}, which here are bounded by 
\[ \sum_{x<mn\le x+x^{1-\delta}} |s(m)| \frac{|\ell(n)|}{n^{\eta}} 
+ \int_0^{\eta} \sum_{x<mkn \le x+x^{1-\delta}} |s(m)| \frac{|\Lambda_\ell(k)|}{k^{\alpha}} \frac{|\ell(n)|}{n^{2\eta+\alpha}} d\alpha. 
\]
Using Lemma \ref{Shiu}, the first term is 
$$ 
\ll_{\kappa} \frac{x^{1-\delta}}{\log x} \exp\Big( \sum_{p\le y} \frac{|f(p)|}{p} + \sum_{y<p\le x} \frac{|f(p)|}{p^{1+\eta}}\Big) 
\ll \frac{x^{1-\delta}}{\log x} (\log y)^{\kappa}.
$$ 
To bound the second term, we distinguish the cases when $mn \le 2\sqrt{x}$ and when $mn>2\sqrt{x}$ so that $k\le 2\sqrt{x}$.  In 
the first case, we use the Brun--Titchmarsh inequality to estimate the sum over $k$, and thus such terms contribute 
$$ 
\ll_{\kappa} \int_0^{\eta} \sum_{mn \le 2\sqrt{x}} |s(m)| \frac{|\ell(n)|}{n^{2\eta+\alpha}} \frac{x^{1-\delta}}{mn}\Big(\frac{mn}{x}\Big)^{\alpha} d\alpha \ll 
\frac{x^{1-\delta}}{\log x} \exp\Big( \sum_{p\le y} \frac{|f(p)|}{p^{1-\alpha}} + \sum_{y< p\le x}\frac{|f(p)|}{p^{1+2\eta}} \Big),  
$$
which is $\ll_{\kappa} x^{1-\delta} (\log y)^{\kappa}/\log x$.  In the second case, we sum over $mn$ first using Lemma \ref{Shiu} and then sum over $k$; thus, we obtain 
\begin{align*}
&\ll_{\kappa} \int_0^{\eta} x^{-\alpha} \sum_{k\le 2\sqrt{x}} \Lambda(k) \sum_{x/k \le mn \le (x+x^{1-\delta})/k} |s(m)| m^{\alpha} \frac{|\ell(n)|}{n^{2\eta}} d\alpha  
\\
&\ll \int_0^{\eta} x^{-\alpha} \sum_{k\le 2\sqrt{x}} \frac{ \Lambda(k)}{k} \frac{x^{1-\delta}}{\log x} \exp\Big( \sum_{p\le y} \frac{|f(p)|}{p^{1-\alpha}} + \sum_{y<p\le x} \frac{|f(p)|}{p^{1+2\eta}}\Big) 
d\alpha \ll \frac{x^{1-\delta}}{\log x} (\log y)^{\kappa}. 
\end{align*} 
This completes our proof of Theorem \ref{ShortHal}.  
 \end{proof} 

To deduce Corollary \ref{ShortHal2} (and also in the proof of Theorem \ref{LipThm}), we will find useful the following 
variant of the maximum modulus principle, which may be found as Lemma 2.2 of \cite{GSDecay}.  

\begin{lemma} \label{lem3.1} Let $a_n$ be a sequence of complex numbers such that $A(s) = \sum_{n=1}^{\infty} a_n n^{-s}$ 
is absolutely convergent for Re$(s)\ge 1$.  For all real numbers $T\ge 1$ and all $0\le \alpha \le 1$ we have 
$$ 
\max_{|y|\le T} |A(1+\alpha+iy)| \le \max_{|y|\le 2T} |A(1+iy)| + O\Big(\frac
\alpha T \sum_{n\geq 1} \frac{|a_n|}{n}\Big), 
$$
and for any $w\ge 1$ 
$$ 
\max_{|y| \le T} |A(1+\alpha+iy) (1-w^{-\alpha- iy} ) | \le \max_{|y| \le 2T} |A(1+iy) (1-w^{-iy}) | + O\Big(\frac{\alpha}{T} \sum_{n=1}^{\infty} \frac{|a_n|}{n} \Big). 
$$ 
\end{lemma} 

\begin{proof} [Proof of Corollary \ref{ShortHal2}]  Take $a_n=f(n)/n^{1/\log x}$, and $T=\frac 12 x^{\delta} (\log x)^{\kappa}$ in Lemma 
\ref{lem3.1}, so that  for $1/\log x\leq \sigma\leq 1$ 
\begin{equation} \label{MaxModBd}
\max_{|t|\le T} |F(1+\sigma+it)| \le \max_{|t|\le 2T} |F(c_0+it)| + O(1),
 \end{equation}
where $c_0=1+ 1/{\log x}$ as before. Combining this with the upper bound $| F(1+\sigma+it)|\leq  \zeta(1+\sigma)^\kappa \ll 1/\sigma^\kappa$ in appropriate ranges of $\sigma$, we deduce Corollary \ref{ShortHal2} from Theorem \ref{ShortHal} (in much the same way as we deduced Corollary \ref{HalCor} from Theorem \ref{GenHal}).
\end{proof}

\subsection{Arithmetic progressions} 

 In this section we begin our study of the mean value of $f$ in an arithmetic progression $a\pmod q$ with $(a,q)=1$.  Our 
 starting point is the orthogonality relation of characters 
  \begin{equation} \label{CharDecomp}
 \sum_{\substack{n\leq x\\ n\equiv a \pmod q}} f(n)= \frac 1{\phi(q)} \sum_{\chi \pmod q}	\chi(a) \sum_{n\leq x} f(n) \overline\chi(n), 
 \end{equation}
 and now we may invoke our earlier work in understanding the mean-values of the multiplicative functions $f(n) \overline{\chi}(n)$.  
 
 \begin{proof}[Proof of Theorem \ref{GenHalq}]  Starting with \eqref{CharDecomp}, we take $T=q^2 (\log x)^{\kappa+2}$ and $y=T^2$, and apply Lemma \ref{keyidr1} to each sum $\sum_{n \leq x} f(n) \overline\chi(n)$.  After summing over $\chi \bmod q$, the second and third terms in \eqref{keyidr3} contribute 
\begin{equation} \label{ErrAP1} 
\le \sum_{\substack{mn \le x \\ mn\equiv a \pmod q}} |s(m)| \frac{|\ell(n)|}{n^{\eta}} + \int_0^{\eta} \sum_{\substack{mkn \le x \\ mkn\equiv a \pmod q}}|s(m )|\frac{|\Lambda_\ell(k)|}{k^{\alpha}}  \frac{|\ell(n)|}{n^{2\eta+\alpha}} d\alpha . 
\end{equation}
We bound these terms in a similar manner to the argument in the previous section.  Thus, by Lemma \ref{Shiu}, the first term in \eqref{ErrAP1} is 
$$ 
\ll \frac{x}{\phi(q) \log x} \exp\Big(\sum_{p\le y} \frac{|f(p)|}{p} + \sum_{y< p\le x} \frac{|f(p)|}{p^{1+\eta}} \Big) 
\ll_{\kappa} \frac{x}{\phi(q) \log x} (\log y)^{\kappa}. 
$$ 
To bound the second term in \eqref{ErrAP1}, we distinguish the cases when $mn \le \sqrt{x}$ and when $mn>\sqrt{x}$ so that $k\le \sqrt{x}$.   
In the first case, we use the Brun--Titchmarsh inequality to estimate the sum over $k$ (noting that $k$ must lie in a fixed progression 
modulo $q$), and obtain  
$$ 
\ll \int_{0}^{\eta} \sum_{mn \le \sqrt{x}} |s(m)| \frac{|\ell(n)|}{n^{2\eta+\alpha}} \Big( \frac{x}{mn} \Big)^{1-\alpha} \frac{1}{\phi(q)} d\alpha  \ll 
\frac{x}{\phi(q) \log x} \exp\Big( \sum_{p\le y} \frac{|f(p)|}{p^{1-\alpha}}  + \sum_{y< p\le x} \frac{|f(p)|}{p^{1+2\eta}}\Big), 
$$ 
which is $\ll_{\kappa} x (\log y)^{\kappa}/(\phi(q)\log x)$.  Finally in the second case when $mn >\sqrt{x}$, we first bound the sum over $mn$ using Lemma 
\ref{Shiu} and then perform the sum over $k$ (which is now at most $\sqrt{x}$).   Thus such terms contribute (multiplying through by $(mn/\sqrt{x})^\alpha$ since $mn >\sqrt{x}$) 
\begin{align*}  
&\ll_{\kappa} \int_0^{\eta} x^{-\alpha/2} \sum_{k\le \sqrt{x}} \frac{\Lambda(k) }{k^\alpha} \sum_{\substack{ mn \le x/k \\ mn k \equiv a\pmod q}} |s(m)| m^{\alpha} \frac{|\ell(n)|}{n^{2\eta}} 
d\alpha \\ 
&\ll \int_0^{\eta}x^{-\alpha/2} \sum_{k\le \sqrt{x}} \frac{\Lambda(k)}{k} \frac{x}{\phi(q) \log x} \exp\Big( \sum_{p\le y} \frac{|f(p)|}{p^{1-\alpha}} + 
\sum_{y<p \le x} \frac{|f(p)|}{p^{1+2\eta}} \Big) d\alpha 
\ll \frac{x}{\phi(q)\log x} (\log y)^{\kappa}. 
\end{align*} 
 These bounds are all acceptable for Theorem \ref{GenHalq}.  
 
 Now we turn to the corresponding contour integrals in \eqref{keyidr2}.     Arguing as in Section \ref{sec:Tailor}, 
we may truncate the corresponding integrals at height $T$, incurring an error term in the analogue of Lemma \ref{keyid3} of   
\[
 \ll \frac{x(\log x)^{\kappa}}{T} \ll \frac x{\phi(q)\log x}, 
 \]
which again is acceptable for Theorem \ref{GenHalq}.
After summing over $\chi$ mod $q$, we are left with a main term (analogously to Proposition \ref{xkeyidr1}) of   
 \[
\begin{split}
 \frac 1{\phi(q)}  \sum_{\chi \pmod q}  	\chi(a)
\int_0^{\eta}& \int_0^{\eta}  \frac{1}{\pi i } \int_{c_0-iT}^{c_0+iT} \eS_{\overline\chi}(s-\alpha-\beta) \\
&\times \eL_{\overline\chi}(s+\beta) \sum_{y<m<x/y} \frac{\Lambda_\ell(m)\overline{\chi}(m)}{m^{s-\beta}}  \sum_{y<n<x/y} \frac{\Lambda_\ell(n)\overline{\chi}(n)}{n^{s+\beta}}  \frac{x^{s-\alpha-\beta}}{s-\alpha-\beta} ds d\beta d\alpha ,
\end{split}
\]
where  $\eS_{\overline \chi}$ and $\eL_{\overline \chi}$ are appropriate twists of $\eS$ and $\eL$ respectively.  To bound this we need 
the following extension of Lemma \ref{MeanSquarePrimes}. 
\begin{lemma} 
\label{lem3.2}    Let $T\ge q^{1+\epsilon}$ be given.  Then for any
complex numbers $a(n)$ we have 
\begin{equation} \label{MontHal2}
\begin{split}
\int_{-T}^{T}  \frac 1{\phi(q)} \sum_{\chi \pmod q}  \Big| \sum_{T^2 \le n\le x} \frac{a(n)\chi(n)\Lambda(n)}{n^{it}} \Big|^2 dt 
& \ll_{\epsilon} \frac 1{\phi(q)} \sum_{T^2 \le n \le x} n|a(n)|^2 \Lambda(n). 
\end{split}
\end{equation}
\end{lemma} 
\begin{proof}  Arguing as in Lemma \ref{MeanSquarePrimes}, the left side of \eqref{MontHal2} may be bounded 
by 
$$ 
 \ll T\sum_{T^2 < m \le x} |a(m)|^2 \Lambda(m) \sum_{\substack{n:\ |n-m| \ll m/T \\ n\equiv m \pmod q}} \Lambda(n), 
 $$ 
 and the lemma follows upon using the Brun--Titchmarsh theorem 
 to bound the number of primes in an interval that are also in a given arithmetic progression.   
 \end{proof}  

We can now proceed much as in the proof of Hal\'{a}sz's Theorem \ref{GenHal} (replacing $\overline{\chi}$ by $\chi$ for convenience).  The contribution from any $\alpha$, $\beta$ fixed is 
(arguing as in  \eqref{sbound})
\begin{align*}
& \ll  x^{1-\alpha-\beta}  \Big( \max_{\substack{\chi \pmod q \\ |t|\le T}}\frac{ |F_\chi(c_0+\beta+it)|}{|c_0+\beta+it|} \Big) \\
&\hskip 1 in \times \frac 1{\phi(q)} \sum_{\chi \pmod q} \int_{-T}^{T}  
 \Big| \sum_{y< n< x/y} \frac{\Lambda_\ell(n)\chi(n)}{n^{c_0+\beta+it}} \Big| \Big| \sum_{y< m < x/y} \frac{\Lambda_\ell(m)\chi(m)}{m^{c_0-\beta+it}} \Big| dt . 
 \end{align*}
Using Cauchy--Schwarz and \eqref{MontHal2} (which is permissible since $y=T^{2}$ and $T=q^2(\log x)^{\kappa+2}$), the second factor above 
is bounded by 
\begin{equation} 
\ll \frac 1{\phi(q)} \Big( \sum_{y< m < x/y} \frac{\Lambda(m)}{m^{1-2\beta}}\Big)^{\frac 12} \Big( \sum_{y< n < x/y} 
\frac{\Lambda(n)}{n^{1+2\beta}} \Big)^{\frac 12} \ll \frac {x^{\beta}}{\phi(q)} \min \Big( \log x, \frac{1}{\beta} \Big). 
\end{equation}
by \eqref{sumprimebeta}.
Integrating over $\alpha, \beta \in [0,\eta]$, and taking $\sigma=\beta+1/\log x$ with $c_0 = 1+1/\log x$  
so that $c_0+\beta=1+\sigma$, we obtain  a bound of 
  $$ 
  \ll  \frac 1{\phi(q)} \frac{x}{\log x} \int_{1/\log x}^{1} \Big( \max_{\substack{ \chi \pmod q \\ |t| \le q^2(\log x)^{\kappa + 2} }} \Big| \frac{F_\chi(1+\sigma+it)}{1+\sigma+it}\Big| 
  \Big) \frac{d\sigma}{\sigma} 
  $$  
for our main term.  Now $|F_\chi(1+\sigma+it)|\ll 1/\sigma^\kappa$ and so the contribution to the integral from any $|t|\geq (\log x)^\kappa$ is $\ll 1$, which is an acceptable error, so we may restrict the maximum to the range  $|t| \leq (\log x)^{\kappa}$. The result follows.
 \end{proof}  

\begin{proof}[Proof of Corollary \ref{HalCorq}] This follows from Theorem \ref{GenHalq}, just as 
Corollary  \ref{HalCor} follows from Theorem \ref{GenHal}.
\end{proof}




%

\section{A Lipschitz estimate for mean-values: Proof of Theorem \ref{LipThm}}  


\noindent Let $x$ be large, and $c_0 = 1+1/\log x$.  Recall that $t_1= t_1(x)$ is chosen such that $|F(c_0 +i t_1)|$ is a maximum 
in the range $|t_1 |\le (\log x)^{\kappa}$.   Given $1\le w \le x^{\frac 13}$, our goal in Theorem \ref{LipThm} is to bound 
\begin{equation} 
\label{4.1} 
\frac{1}{x^{1+it_1}} \sum_{n\le x} f(n)  - \frac{1}{(x/w)^{1+it_1}} \sum_{n\le x/w} f(n). 
\end{equation}

We apply Lemma \ref{keyidr1}  twice, with $x$ and $x/w$, and divide through by $x^{1+it_1}$ and $(x/w)^{1+it_1}$ respectively.  
 The terms arising from the second and third sums in \eqref{keyidr3} may be bounded (as in Lemma \ref{errors23}) by 
 $O((\log y)^{\kappa}/\log x)$. 
The integral arising from \eqref{keyidr2} is 
$$ 
\int_0^{\eta}\int_0^{2\eta} \frac{1}{2\pi i} \int_{c-i\infty}^{c+i\infty} {\mathcal S}(s) {\mathcal L}(s+\alpha+\beta) 
\frac{\mathcal L^{\prime}}{\mathcal L}(s+\alpha) \frac{\mathcal L^{\prime}}{\mathcal L} (s+\alpha+\beta) \frac{x^{s-1-t_1}-(x/w)^{s-1-t_1}}{s} 
ds d\beta d\alpha. 
$$ 
We now tailor this integral as in Section 2.2, taking there $T = (\log x)^{\kappa +1}$, as before, and setting $y= \max(ew, T^2)$.  
Then, up to an error term $O((\log y)^{\kappa}/(\log x))$ 
which is acceptable for Theorem \ref{LipThm}, we may bound the magnitude of \eqref{4.1} by 
\begin{align} 
\label{4.2} 
\ll \int_0^\eta \int_0^{\eta} \Big|& \int_{c_0-iT}^{c_0+iT} {\mathcal S}(s-\alpha- \beta) {\mathcal L}(s+\beta) \nonumber \\ &\times 
\sum_{\substack{ y< m < x/y \\ y< n< x/y} } \frac{\Lambda_{\ell}(m)}{m^{s-\beta}} \frac{\Lambda_\ell(n)}{n^{s+\beta}} 
\frac{x^{s-\alpha-\beta-1-it_1}-(x/w)^{s-\alpha-\beta-1-it_1}}{s-\alpha-\beta} ds  \Big| d\beta d\alpha. 
\end{align} 
From here on we follow the argument in Section \ref{PfHalasz}, making suitable modifications.  

Thus, arguing as in \eqref{sbound} and \eqref{HalRev9}, we may bound (for given $\alpha$, $\beta$) the inner integral 
in \eqref{4.2} by 
\begin{equation} 
\label{4.3} 
\ll x^{-\alpha} \min \Big( \log x, \frac{1}{\beta} \Big) \max_{|t| \le T} \frac{|F(c_0 + \beta+it)(1-w^{\alpha+\beta + 1 +it_1 -c_0 -it})|}{|c_0+\beta+it|}.
\end{equation} 
Using the triangle inequality and since $\alpha$ and $\beta$ are below $\eta \le 1/\log (ew)$, we may bound the maximum in \eqref{4.3} by 
\begin{align*}
&\ll \max_{|t| \le T} \frac{|F(c_0+\beta+it)|}{|c_0+\beta+it|} \Big( |1-w^{-\beta+it_1-it}| + |w^{-\beta+it_1 -it} - w^{\alpha+\beta+1-c_0+it_1-it}|\Big) \\ 
&\ll \max_{|t| \le T} \frac{|F(c_0+\beta+it)(1-w^{-\beta+it_1-it})|}{|c_0+\beta+it|} + \Big(\alpha +\beta+\frac{1}{\log x}\Big) (\log w)\Big(\beta +
\frac{1}{\log x}\Big)^{-\kappa}. 
\end{align*} 
Insert this bound in \eqref{4.3}, and then into \eqref{4.2}; we conclude that the quantity in \eqref{4.2} is bounded by 
\begin{equation} 
\label{4.4} 
\frac{1}{\log x} \int_0^{\eta} \min \Big( \log x, \frac{1}{\beta} \Big) 
\Big(  \max_{|t|\le T} \frac{|F(c_0+\beta +it)  |1-w^{-\beta -it+it_1}|}{|c_0+\beta+it|}  + \log w \Big(\frac{1}{\log x} +\beta\Big)^{-\kappa +1} \Big) d\beta. 
\end{equation} 
A small calculation allows us to bound the contribution of the second term above by 
\begin{equation} 
\label{4.5} 
\ll \frac{\log w}{\log x}\,  (\log x)^{\kappa -1} \,  \log\Big( \frac{\log x}{\log ew}\Big).
\end{equation} 
(The $\log ( \frac{\log x}{\log ew} )$ term can be replaced by $1$ except when $\kappa=1$.)

It remains to handle the first term in \eqref{4.4}, for which we require the following lemma.

\begin{lemma}\label{lem4.1}   Let $f\in {\mathcal C}(\kappa)$. 
Select a real number  $t_1$ with  $|t_1| \le (\log x)^{\kappa}$ which maximizes $|F(1+1/\log x+it_1)|$.  
Then for $|t|\le(\log x)^{\kappa}$ we have 
$$
|F(1+1/\log x + it)| \ll_{\kappa} (\log x)^{\frac{2\kappa}{\pi} } \Big( \frac{\log x}{1+|t-t_1|\log x} + (\log \log x)^2 \Big)^{\kappa(1-\frac{2}{\pi})}. 
$$ 
If $f(n)\geq 0$ for all $n$ then we obtain the analogous result with $1/\pi$ in place of $2/\pi$.
\end{lemma}

\begin{proof}   
If $|t-t_1| \le 1/\log x$ then the stated estimate is immediate, and we suppose below that $(\log x)^{\kappa} \gg |t-t_1|\ge 1/\log x$.  
By the maximality of $t_1$, we have $|F(1+1/\log x + it)| \le |F(1+1/\log x +it_1)|$, and so, setting $ \tau = |t_1-t|/2$, we obtain  
\begin{align*}
 \log |F(1+1/\log x+it)|  & \leq \frac{1}{2} \Big(\log |F(1+1/\log x + it)| + \log |F(1+1/\log x +it_1)| \Big)   \\ 
 & = \text{ Re}\Big(  \sum_{n\geq 2} \frac{\Lambda_f(n)}{n^{1+1/\log x} \log n} \cdot  n^{i(t+t_1)/2} \cdot \frac{n^{i\tau}+n^{-i\tau} }2 \Big) \\
 & \leq \kappa  \sum_{n\geq 2} \frac{\Lambda(n)}{n^{1+1/\log x} \log n}| \cos (\tau \log n)| \\
& = \kappa  \sum_{2 \leq n \leq x} \frac{\Lambda(n)}{n \log n}| \cos (\tau \log n)| + O(1).
\end{align*}
We may now estimate the sum above using the prime number theorem (in the strong form $\psi(x) =x +O(x\exp(-c\sqrt{\log x}))$) 
and partial summation.   For a suitably large constant $C = C_{\kappa}$, put $Y= \max( \exp(C (\log \log x)^2), \exp(1/\tau))$.  For $n\le Y$ we use $|\cos (\tau\log n)| \le 1$, while for larger $n$ we 
can exploit the fact that $\int_0^1 |\cos (2\pi y)| dy =2/\pi$.   In this way, we obtain (and the reader may consult the proof of Lemma 2.3 of \cite{GSDecay} for further 
details if needed) 
$$ 
\sum_{2\le n \le x} \frac{\Lambda(n)}{n \log n}| \cos (\tau \log n)| \le \frac{2}{\pi} \log \log x + \Big(1-\frac{2}{\pi}\Big)\log \log Y +O(1), 
$$ 
from which the first assertion of the lemma follows.  

 If $f(n)\geq 0$ for all $n$, then using $F(c_0+it)=\overline{F(c_0-it)}$ we find 
 \begin{align*}
 \log |F(1+1/\log x+it)| &= \frac{1}{2} \sum_{n\geq 2} \frac{1}{n^{1+1/\log x} \log n} \text{Re}\left( \Lambda_f(n) (n^{-it} + n^{it}) \right) \\ & \leq  \kappa \sum_{p\geq 2} \frac{\log p \max\{  0, \cos (t \log p)\}}{p^{1+1/\log x} \log p} +O(1).     
 \end{align*}
Now we argue as above, exploiting here that   $\int_0^1 \max\{  0,  \cos (2\pi  y)\}dy = 1/\pi$. 
 \end{proof}

Before applying the lemma, we comment that the estimates given there are in general the best possible.  Suppose $\kappa=1$.   Given $t_1$ with $|t_1|\le \log x$ 
consider $f\in {\mathcal C}(1)$ defined by $f(p) = \text{sign} (\cos(t_1\log p))$.  Here the bound of the lemma is attained for $t = -t_1$.  
Similarly for the non-negative case, consider the function defined by $f(p) =(1+ \text{sign}(\cos (t_1 \log p)))/2$.   For more discussion on these 
examples, see section 9 below.  

Now we return to the first term in \eqref{4.4}, seeking a bound for 
\begin{equation} 
\label{4.6} 
\max_{|t| \le T } \frac{|F(c_0+\beta+it)(1-w^{-\beta-it+it_1})|}{|c_0+\beta+it|}.    
\end{equation} 
%
Note that if $|t| \ge (\log x)^{\kappa}/10$ then the quantity in \eqref{4.6} may be bounded trivially by $\ll 1$. 
Now restrict to the range $|t| \le (\log x)^{\kappa}/10$, and distinguish the cases $|t_1| \le (\log x)^{\kappa}/4$ and $|t_1| > (\log x)^{\kappa}/4$.  
In the second case $|t_1| > (\log x)^{\kappa}/4$, we have 
\begin{align*}
\max_{|t| \le (\log x)^{\kappa}/10} \frac{|F(c_0+\beta+it)(1-w^{-\beta-it+it_1})|}{|c_0+\beta+it|} &\ll \max_{|t| \le (\log x)^{\kappa}/10} |F(c_0+\beta+it)| \\
&\ll \max_{|t| \le (\log x)^{\kappa}/5} |F(1+1/\log x+it)| + \eta \\
&\ll (\log x)^{\frac{2\kappa}{\pi}} (\log \log x)^{2\kappa(1-\frac 2\pi)},
\end{align*} 
upon using Lemma \ref{lem3.1} with $a_n=f(n)/n^{1/\log x}$, and Lemma \ref{lem4.1}.  In the first case when $|t_1| \le (\log x)^{\kappa}/4$, put $A(1+it) = F(c_0+i(t+t_1))$, 
so that using the second part of Lemma \ref{lem3.1} 
\begin{align*}
\max_{|t| \le (\log x)^{\kappa}/10} \frac{|F(c_0+\beta+it)(1-w^{-\beta-it+it_1})|}{|c_0+\beta+it|} &\ll \max_{|t| \le (7/20)(\log x)^{\kappa}} 
|A(1+\beta+it) (1-w^{-\beta -it})|\\ 
& \ll \max_{|t| \le (7/10) (\log x)^{\kappa}} |A(1+it) (1-w^{-it})| + \eta.  
\end{align*} 
Appealing now to Lemma \ref{lem4.1} where $A(1+it)=F(c_0+i(t+t_1))$, the above is bounded by 
\begin{align*}
&\max_{|t| \le (7/10)(\log x)^{\kappa}} \Big( \min (1, |t|\log w) (\log x)^{\frac{2\kappa}{\pi}} \Big( \frac{\log x}{1+|t|\log x} + (\log \log x)^2 \Big)^{\kappa(1-\frac{2}{\pi})} \Big)
\\
&\ll  (\log x)^{\frac{2\kappa}{\pi}} (\log \log x)^{2\kappa(1-\frac{2}{\pi})} + (\log x)^{\kappa} \Big(\frac{\log w}{\log x}\Big)^{\min(1,\kappa(1-\frac{2}{\pi}))}, 
\end{align*}  
after a little calculation.  
Thus in all cases we may bound the quantity in \eqref{4.6} by
  \begin{equation*} 
 \label{4.7}  
\max_{|t| \le T}   \frac{|F(c_0+\beta+it)(1-w^{-\beta-it+it_1})|}{|c_0+\beta+it|}  \ll  (\log x)^{\kappa} 
\Big(\frac{\log w + (\log \log x)^2}{\log x}\Big)^{\min(1,\kappa(1-\frac{2}{\pi}))}.   
\end{equation*}

Using this estimate in \eqref{4.4}, we conclude that the first term there  contributes 
$$ 
\ll \frac{1}{\log x}  (\log x)^{\kappa} \Big(\frac{\log w + (\log \log x)^2}{\log x}\Big)^{\min(1,\kappa(1-\frac{2}{\pi}))} \log \Big( \frac{\log x}{\log ew}\Big) . 
$$ 
The bound in Theorem \ref{LipThm} for general $f \in {\mathcal C}(\kappa)$ follows upon combining this estimate with \eqref{4.5}.  
The bound for non-negative functions $f$ follows similarly, using now the stronger input of Lemma \ref{lem4.1} for that situation.

\section{Repulsion results}\label{sec:Nottoomany}

\noindent   We first show that $|F(c_0+it)|$ can only be large for $t$ lying in a few rare intervals;  
in other words, that large values ``repel'' one another.  Such results are useful in many 
contexts (see for example Lemma 5.1 of \cite{Sou} which formulates a similar result in a slightly more 
general setting); while we will not use Proposition \ref{Rep1} directly in this paper, it will be useful in \cite{II} 
in extracting asymptotic formulae  (such as the examples given in Section 9).


 \begin{proposition} 
 \label{Rep1}  Let $0 \leq \rho \leq 1$. If $t_1,\ldots, t_\ell \in \mathbb R$ satisfy $|t_j| \le (\log x)^{\kappa}$, and 
 $|t_j-t_k| \ge  1/(\log x)^\rho$ for all $j\neq k$,  and  $\delta_1,\ldots, \delta_\ell\geq  0$ then 
 \[
 \sum_{j=1}^{\ell} \log |F(1+1/\log x +\delta_j+it_j)| \le \kappa ( \ell+\ell(\ell-1) \rho)^{1/2}  \log\log x + O_{\ell, \kappa}(\log \log \log x) . 
 \]
 \end{proposition}

 \begin{proof}[Proof of Proposition \ref{Rep1}] The proof uses a simple but powerful ``duality'' idea that is fundamental in large values theory for Dirichlet polynomials.
Since each $|\Lambda_f(n)|\leq \kappa \Lambda(n)$, we have
 \begin{align*} 
  \sum_{j=1}^{\ell} \log |F(1+1/\log x + \delta_j+it_j)| &= \sum_{n\geq 2} \frac{1}{n^{1+1/\log x} \log n} \text{Re}\Big(  \Lambda_f(n) \sum_{j=1}^{\ell} n^{-\delta_j -it_j}\Big)  \\
 &\le \kappa \sum_{n\geq 2} \frac{\Lambda(n)}{n^{1+1/\log x} \log n} \Big| \sum_{j=1}^{\ell} n^{-\delta_j -it_j}\Big|.  
 \end{align*} 
 By Cauchy--Schwarz the square of this sum is
 $$ 
  \le 
  \sum_{n\geq 2} \frac{\Lambda(n)}{n^{1+1/\log x} \log n}  \cdot  \sum_{n\geq 2} \frac{\Lambda(n)}{n^{1+1/\log x}\log n} \Big|\sum_{j=1}^{\ell} n^{-\delta_j -it_j} \Big|^2 .
 $$ 
 The first factor  above is $\log \log x + O(1)$.  Expanding out the 
 sum over $j$, the  second factor on the right hand side is 
 $$ 
\sum_{j,k =1}^{\ell} \log |\zeta(1+1/\log x+ \delta_j +\delta_k +i(t_j-t_k))| \le \ell \log \log x + \ell(\ell-1) \rho \log\log x + O(\ell^2 \log \log \log x),
$$ 
where the first term comes from the $j=k$ terms, and the second term from the bound
\[
 \log |\zeta(1+1/\log x+ \delta_j +\delta_k +  i(t_j-t_k))| \le \rho \log\log x + O_{\kappa}(\log \log \log x)
 \]
 whenever $j\ne k$, since  $2 (\log x)^{\kappa} \ge |t_j-t_k| \ge 1/(\log x)^\rho$.   The proposition follows.  
  \end{proof}  


Of greater use in the present paper is Proposition \ref{Rep2} below, which generalizes  Proposition \ref{Rep1} to Dirichlet characters.  
To prove that we require the following lemma.  

\begin{lemma} 
\label{lem5.2}  Let $\chi \pmod q$ be a non-principal character.  Let $x\ge 10$, and suppose that $|t| \le (\log x)^A$ for 
some constant $A$, and $y \ge q \log x$.  Then for all $\sigma \ge 1+1/\log x$ we have 
$$ 
 |L_y (\sigma+it, \chi)| \ll_{A} 1, 
$$ 
where (for Re$(s)>1$) 
$$
L_y(s,\chi) := L(s,\chi) \prod_{p\le y} \Big(1- \frac{\chi(p)}{p^s}\Big) = \prod_{p>y} \Big( 1-\frac{\chi(p)}{p^s}\Big)^{-1}. 
$$ 
\end{lemma}  
\begin{proof} This follows from Lemma 4.1 of Koukoulopoulos \cite{Koukoulopoulos}, where an elementary 
proof based on a sieve argument is given.  

The following alternative argument, which uses some basic analytic properties of $L$-functions, may also be illuminating.   Suppose, 
for simplicity, that $\chi$ is primitive, and put ${\mathfrak a}=0$ or $1$ depending on whether $\chi$ is even or odd.  The completed $L$-function 
$\Lambda(s,\chi)  = (q/\pi)^{s/2} \Gamma((s+{\mathfrak a})/2) L(s,\chi)$ extends to an entire function of the complex plane of order $1$ (see Chapter 9 of  \cite{Dav}).   We claim that for fixed $t \in {\Bbb R}$, the size of the completed $L$-function $|\Lambda(\sigma+it, \chi)|$ is monotone increasing in $\sigma >1$.   
To see this, suppose $\sigma_1 \ge \sigma_2 >1$, and note that by Hadamard's 
factorization formula (see (16) and (18) of Chapter 12 of \cite{Dav})
$$ 
\frac{|\Lambda(\sigma_1+it,\chi)|}{|\Lambda(\sigma_2+it, \chi)| } = \prod_{\rho} \Big| \frac{1-(\sigma_1+it)/\rho}{1-(\sigma_2+it)/\rho} \Big|  
= \prod_{\rho} \Big| \frac{\rho -\sigma_1 -it}{\rho -\sigma_2 -it} \Big|, 
$$ 
where the product is over all non-trivial zeros $\rho$ of $L(s,\chi)$, which satisfy $0\le \text{Re}(\rho) \le 1$.   
Since $0\le \text{Re}(\rho) \le 1$ and $\sigma_1 \ge \sigma_2 >1$, each individual term in our product above is $\ge 1$, and our 
claim follows.  

Now if $\sigma \ge 1+1/\log y$ then 
$$ 
|L_y(\sigma+it,\chi)| \ll \prod_{p>y} \Big( 1- \frac{1}{p^{\sigma}}\Big)^{-1} \ll 1, 
$$ 
so that the lemma follows in this case.  If $1+1/\log x \le \sigma \le 1+1/\log y$, then using Stirling's formula and the monotonicity observation above we 
find that 
$$ 
|L_y(\sigma+it,\chi)| \asymp \Big|\frac{L(\sigma+it,\chi)}{L(1+1/\log y+it, \chi)}\Big| \asymp \Big| \frac{\Lambda(\sigma+it,\chi)}{\Lambda(1+1/\log y+it,\chi)} \Big| 
\le 1. 
$$ 
This completes the lemma for primitive $\chi$, and it is a simple matter to extend to the general case.  \end{proof}  

Before stating our proposition, let us recall some notation used previously in Section 3.2.   If $f \in {\mathcal C}(\kappa)$ is given, and 
$y$ is a real parameter, and $\chi \pmod q$ a Dirichlet character, then (with the sum being over all $n$ with prime factors all larger than $y$) 
$$ 
{\mathcal L}_{\chi}(s) = {\mathcal L}_{\chi}(s;y) = \sum_{\substack{ n \\ p(n) > y}} \frac{f(n)}{n^s} \chi(n). 
$$

\begin{proposition} 
 \label{Rep2}  Let $f\in {\mathcal C}(\kappa)$. Suppose that $\chi_1$, $\ldots$, $\chi_J$ are distinct Dirichlet characters $\pmod q$,  and 
 that   $t_1$, $\ldots$, $t_J$ are real numbers with $|t_j| \le (\log x)^{\kappa}$.   Let $\delta_1$, $\ldots$, $\delta_J$ be non-negative real 
 numbers.    For $x\ge 10$ and $y$ in the range $q\log x \le y \le \sqrt{x}$ we have 
 \[
 \min_{1\leq j\leq J} \log |\eL_{\chi_j}(1+1/\log x + \delta_j +it_j)|  \le  \frac{\kappa}{\sqrt{J}}   \Big( \log\Big( \frac{\log x}{\log y} \Big)+O_{\kappa}( J) \Big).
  \]
   \end{proposition}

 \begin{proof} The proof is similar to that of Proposition \ref{Rep1}.  
 We start with 
 \begin{align*} 
  \sum_{j=1}^{J} \log |\eL_{\chi_j}(1+1/\log x +\delta_j +it_j)| &= \sum_{p(n)>y} 
   \frac{1}{n^{1+1/\log x} \log n} \text{Re}\Big(  \Lambda_f(n) \sum_{j=1}^{J} {\chi_j}(n) n^{-\delta_j -it_j}\Big)  \\
 &\le \kappa \sum_{p(n)>y} \frac{\Lambda(n)}{n^{1+1/\log x} \log n} \Big| \sum_{j=1}^{J} {\chi_j}(n) n^{-\delta_j - it_j}\Big|.  
 \end{align*} 
 By Cauchy--Schwarz the square of the sum here is 
\begin{equation} \label{CS-conseq} 
 \le   \sum_{n>y}  \frac{\Lambda(n)}{n^{1+1/\log x} \log n}  \cdot \sum_{p(n)>y}  \frac{\Lambda(n)}{n^{1+1/\log x}\log n} \Big|\sum_{j=1}^{J} {\chi_j}(n)
 n^{-\delta_j - it_j} \Big|^2  
\end{equation} 
 The first factor is $\log ( \frac{\log x}{\log y} )  + O(1)$.  Expanding out the 
 sum over $j$, the  second factor   is 
 $$ 
\sum_{j, k =1}^{J} \log |L_y(1+1/\log x+\delta_j + \delta_k +i(t_j-t_k), \chi_j\overline{\chi_k})|.    
$$ 
The terms $j=k$ contribute $ \le J \log (\frac{\log x}{\log y}) + O(J)$, while by Lemma \ref{lem5.2} the terms $j\neq k$ 
contribute $O_{\kappa}(J^2)$.   We conclude that 
$$ 
  \sum_{j=1}^{J} \log |\eL_{\chi_j}(1+1/\log x +\delta_j +it_j)|  \le \kappa \sqrt{J} \Big( \log \Big(\frac{\log x}{\log y} \Big) + O_{\kappa}(J) \Big), 
$$ 
and the proposition follows.  
 \end{proof}  



 

 \section{Multiplicative functions in arithmetic progressions and the pretentious large sieve}\label{secapsandpls}



\noindent  Given $f \in {\mathcal C}(\kappa)$ and a Dirichlet character $\chi \pmod q$, recall that we defined 
$$ 
S_f(x, \chi) = \sum_{n\le x} f(n) \overline{\chi(n)}. 
$$ 
The proof of Hal{\' a}sz's theorem in arithmetic progressions presented in Section 3.2 may be 
refined to take into account the contribution of a given set of characters.  Recall that for a character $\chi \pmod q$ we 
defined $M_{\chi} = M_{\chi}(x)$ by the relation 
$$ 
\max_{|t| \le (\log x)^{\kappa} } \Big| \frac{F_{\chi}(1+1/\log x+ it)}{1+1/\log x + it } \Big| = e^{-M_{\chi}} (\log x)^{\kappa}. 
$$ 

 \begin{proposition}
 \label{HalCorq2}  Let $\mathcal X$ be a set of characters $\pmod q$ and put 
 $$
 M=M_{\mathcal X}(x) := \min_{\chi\not\in {\mathcal X}} M_{\overline \chi}(x).   
 $$
Then, with  
$$ 
E_{f,{\mathcal X}}(x;q;a):=\sum_{\substack{n\leq x\\ n\equiv a \pmod q}}  f(n) - \frac 1{\phi(q)} \sum_{\chi \in {\mathcal X}} \chi(a) S_f(x,\chi), 
$$ 
we have 
$$ 
E_{ f,{\mathcal X}}(x;q;a) \ll \frac 1 {\phi(q)}  \frac{x}{\log x}  \big(   (1+M)e^{-M} (\log x)^{\kappa} + |\mathcal X|(\log (q\log x))^{\kappa}     \big) . 
$$ 
\end{proposition} 

%

\begin{proof}  We start with the orthogonality relation \eqref{CharDecomp}, omitting the characters in ${\mathcal X}$, so that 
$$ 
E_{f,{\mathcal X}}(x;q;a) = \frac{1}{\phi(q)} \sum_{\substack{\chi \pmod q \\ \chi \not\in {\mathcal X}}} \chi(a) \sum_{n\le x} f(n) \overline{\chi}(n), 
$$ 
and then apply Lemma \ref{keyidr1} to each of the inner sums taking $T = q^2 (\log x)^{\kappa+2}$ and $y=T^2$.   Consider the contribution of the 
second term  arising from \eqref{keyidr3}, which is 
$$ 
\frac{1}{\phi(q)} \sum_{\substack{\chi \pmod q \\ \chi \not\in {\mathcal X}}} \chi(a) \sum_{mn \le x} s(m) \frac{\ell(n)}{n^{\eta}}\overline{ \chi}(mn)  
= \! \sum_{\substack{mn \le x\\ mn \equiv a \pmod q}} s(m) \frac{\ell(n)}{n^{\eta}} +O\Big(\frac{ |{\mathcal X}|}{\phi(q)} \sum_{mn \le x} |s(m)| \frac{|\ell(n)|}{n^{\eta}}\Big).  
$$ 
A similar decomposition holds for the third term arising from \eqref{keyidr3}.  The first term above together with its counterpart arising from the 
third term in \eqref{keyidr3} contribute the expression in \eqref{ErrAP1}, and may be bounded as before by $\ll x(\log y)^{\kappa}/(\phi(q)\log x)$.  
Similarly, the second term above together with its counterpart may be handled by Lemma \ref{errors23}, and is bounded by $|{\mathcal X}| x (\log y)^{\kappa}/(\phi(q)\log x)$.  
From here on, we may follow the argument of Section 3.2, and making small modifications to that argument yields the proposition.  
\end{proof} 


The following result contains Theorem \ref{thm1.8}. 

\begin{theorem} 
 \label{thm6.1}  Let $q$ be a natural number, and let $J \ge 1$ be a given integer.  
 Let $X$ be large with $X\ge q^{10}$, and set $Q = q \log X$.  Consider the following three definitions 
 of a set of $J$ exceptional characters: 
  
 (i)  The set ${\mathcal X}_J^1$ consists of the characters $\chi$ corresponding to the $J$ largest values of $\max_{\sqrt{X} \le x\le X} |S_f(x,\chi)|/x$.  
 
 (ii)  The set ${\mathcal X}_J^2$ consists of the characters $\chi$ corresponding to the $J$ smallest values of $M_{\overline{\chi}}(X)$.   
 
 (iii)  The set ${\mathcal X}_J^3$ consists of the characters $\chi$ corresponding to the $J$ largest values of $\max_{|t|\le (\log X)^{\kappa}} |\eL_{\overline{\chi}}(1+1/\log X +it)|$; here take $y=q^4(\log X)^{2\kappa +4}$  in the definition of $\eL_{\overline{\chi}}$.  
 
\noindent  In all of the cases $\ell=1$, $2$, or $3$ one has, uniformly for all $x$ in the range $\sqrt{X} \le x \le X^2$, 
   \begin{equation}\label{MainAPs}
E_{f,{\mathcal X}_J^{\ell}}(x;q;a) \ll_{\kappa,J}     
   \Big(   \frac{\log  Q}{\log x} \Big)^{ \kappa (1 - \frac{1}{\sqrt{J+1}} )} \log\Big( \frac{\log x} {\log Q}\Big) \cdot \frac x{\phi(q)}   (\log x)^{\kappa-1}   .
\end{equation}
Further, for any character $\psi \not\in {\mathcal X}_J^1 \cap {\mathcal X}_J^2 \cap {\mathcal X}_J^3$ one has, uniformly for all $x$ in the range $\sqrt{X} \le x \le X^2$,  
\begin{equation} 
\label{6.2} 
|S_f(x,\psi)|  \ll_{\kappa, J}   \Big(   \frac{\log  Q}{\log x} \Big)^{ \kappa (1 - \frac{1}{\sqrt{J+1}} )} \log\Big( \frac{\log x} {\log Q}\Big) \cdot x   (\log x)^{\kappa-1} . 
\end{equation} 
 \end{theorem}


\begin{proof}  By Proposition \ref{Rep2}, there are at most $J$ characters $\chi$ with 
$$ 
\max_{|t| \le (\log X)^{\kappa}} \log |{\eL_{\overline{\chi}}} (1+1/\log X+ it)| > \frac{\kappa}{\sqrt{J+1}} \Big( \log \Big(\frac{\log X}{\log y}\Big) + O(J)\Big).   
$$ 
For $\sqrt{X} \le x \le X^2$, it is easy to check that the difference between $\log |\eL_{\overline{\chi}}(1+1/\log x +it)|$ and $\log |\eL_{\overline{\chi}}(1+1/\log X+it)|$ is 
$O(1)$.   Therefore, we have that there are at most $J$ characters $\chi$ with 
\begin{equation*} 
\label{6.3} 
\max_{\sqrt{X} \le x\le X^2} \max_{|t| \le (\log X)^{\kappa} } \log |\eL_{\overline{\chi}}(1+1/\log x+it)|  > \frac{\kappa}{\sqrt{J+1}} \Big( \log \Big( \frac{\log X}{\log Q}\Big) + O(J)\Big).   
\end{equation*} 
For a character $\psi$ different from these $J$ characters, one has 
\begin{align}
\max_{\sqrt{X} \le x\le X^2} \max_{|t| \le (\log X)^{\kappa}} \Big| \frac{F_{\overline{\psi}}(1+1/\log x+it)}{1+1/\log x+it}\Big| 
&\ll   (\log y)^{\kappa}\max_{|t| \le (\log X)^{\kappa}} |\eL_{\overline{\psi}}(1+1/\log X+it)| \nonumber \\
&\ll_{\kappa,J}  (\log Q)^{\kappa} \Big(\frac{\log X}{\log Q}\Big)^{\kappa/\sqrt{J+1}}. \label{logQPower}
\end{align} 
In other words, for such a character $\psi$ one has 
$$ 
\min_{\sqrt{X} \le x \le X^2} M_{\overline{\psi}}(x)  \ge \kappa \Big(1- \frac{1}{\sqrt{J+1}} \Big) \log \Big(\frac{\log X}{\log Q} \Big) + O_{\kappa,J}(1). 
$$ 
The estimate \eqref{MainAPs} for $\ell=2$ and $3$ (that is, cases (ii) and (iii)) follows now from Proposition \ref{HalCorq2}, noting that the contribution from the $(\log Q)^\kappa$ error term in 
Proposition \ref{HalCorq2} is smaller than the upper bound in \eqref{MainAPs}.

If $\psi \not \in {\mathcal X}_J^{2} \cap {\mathcal X}_J^{3}$, then applying Corollary \ref{HalCor} (the first corollary of Hal{\' a}sz's Theorem \ref{GenHal}) to 
$f\overline{\psi}$  yields \eqref{6.2}.   By the maximality in the definition of ${\mathcal X}_J^1$, we deduce that 
the estimate \eqref{6.2} must also hold for any character $\psi \not \in {\mathcal X}_J^1$.



Now we handle \eqref{MainAPs} in case (i).  We have already shown the result for the set ${\mathcal X}_J^3$.  Taking the difference between the sums defining $E_{f,{\mathcal X}_J^1}(x;q,a)$ and $E_{f,{\mathcal X}_J^3}(x;q,a)$ yields that
$$ 
|E_{f,{\mathcal X}_J^1}(x;q,a) - E_{f,{\mathcal X}_J^3}(x;q,a)| \le \frac{1}{\phi(q)} \Big( \sum_{\chi \in {\mathcal X}_J^1 \setminus {\mathcal X}_J^3} + 
\sum_{\chi\in {\mathcal X}_J^3 \setminus {\mathcal X}_J^1} \Big) |S_f(x;\chi)|. 
$$ 
Using \eqref{6.2}, this yields a bound on the difference  that is acceptable, and so the theorem follows in case (i) also. 
\end{proof} 

\begin{proof}[Proof of Theorem \ref{PLSthm}]
Let ${\mathcal X}$ be a set of characters $\pmod q$, and consider 
$$ 
\sum_{(a,q)=1}   \Big| E_{f,{\mathcal X}}(x;q,a) \Big|^2= \sum_{(a,q)=1} \Big|\frac{1}{\phi(q)} \sum_{\chi \not\in {\mathcal X}} \chi(a) S_f(x,\chi)\Big|^2. 
$$ 
Expanding out the inner sum, we obtain 
$$ 
\frac{1}{\phi(q)^2} \sum_{\chi , \psi \not \in {\mathcal X} } S_f(x,\chi) \overline{S_f(x,\psi)} \sum_{a \pmod q} \chi(a)\overline{\psi(a)} 
= \frac{1}{\phi(q)} \sum_{\chi \not \in {\mathcal X}} |S_f(x,\chi)|^2. 
$$ 
Thus
\begin{align} \label{PLS-id}
 \frac 1{\phi(q)} \sum_{\chi\not \in {\mathcal X}}	|S_f(x,\chi) |^2 = \sum_{(a,q)=1} |E_{f,{\mathcal X}}(x;q,a) |^2 .
\end{align}
Taking ${\mathcal X}$ to be the set ${\mathcal X}_J^1$ of Theorem \ref{thm6.1}, and inserting \eqref{MainAPs} into \eqref{PLS-id}, we obtain the result.  
\end{proof}




\section{Counting primes in short intervals and arithmetic progressions}

\noindent Here and in the next section, we wish to understand  (given $f\in {\mathcal C}(\kappa)$) 
\begin{equation}\label{finsubsets}
\frac 1{| \mathcal N|}\ \sum_{n\in \mathcal N} \Lambda_f(n), 
\end{equation}
where $\mathcal N$ is either a short interval near $x$, or an arithmetic progression up to $x$ (with uniformity in the modulus $q$).
One can relate the prime sums in \eqref{finsubsets} to multiplicative functions as follows.  Set $G(s) =1/F(s)$, and let $g(n)$ 
denote the corresponding multiplicative function (so that $g(p)=-f(p)$ for all primes $p$).   Note that $G'/G=-F'/F$ and so $\Lambda_g(n)=-\Lambda_f(n)$,  whence $g$ also belongs to ${\mathcal C}(\kappa)$.   Comparing Dirichlet series coefficients in the identity $-F^{\prime}/F  = (1/F) \cdot (-F^{\prime})  = G \cdot (-F^{\prime})$, we obtain 
 \[
\Lambda_f(n)  = \sum_{k\ell=n}  g(k) f(\ell)\log \ell.
\]
We can try to understand averages of the RHS over $n \in {\mathcal N}$ using the hyperbola method, noting 
that the sums over $k$ and $\ell$ now involve multiplicative functions.   However, this approach is doomed to failure because the 
Hal{\' a}sz type results allow us to save at most one logarithm (at least when $\kappa = 1$), whereas the trivial bound obtained by taking absolute values in 
our expression above for $\Lambda_f(n)$ with $f\in {\mathcal C}(1)$ is, (in the particular case that $f=1$ and $g=\mu$, for simplicity),
$$ 
\frac{1}{|{\mathcal N}|} \sum_{k \ell \in {\mathcal N}} |\mu(k) \log \ell  | \asymp (\log x)^2. 
$$ 
To overcome this hurdle, we perform a ``pre-sieving" to eliminate integers with prime factors below $y$ for a 
suitable $y$, and reprove several of our earlier results, though now restricted to multiplicative functions supported only on large prime factors.

\subsection{Multiplicative functions supported only on the primes $>y$}

 Suppose $f \in {\mathcal C}(\kappa)$ with $f(p^k) = 0$ for all $p \le y$ and $k\ge 1$.   
 Thus ${\mathcal S}(s) = 1$ and ${\mathcal L}(s) = F(s)$, and 
 making small modifications to the proof of  Proposition \ref{xkeyidr1} (specifically in the analogues of Lemmas \ref{errors23} and \ref{keyid3}),  we arrive at the following result.    
 
\begin{proposition} 
\label{xkeyidr1y} Given $\sqrt{x} >y\geq 10$, let $\frac 12 > \eta =\frac 1{\log y}>0$ and $c_0=1+\frac 1{\log x}$.   Let $f \in {\mathcal C}(\kappa)$ with $f(p^k)=0$ for all primes $p\leq y$ and $k\ge 1$.  Then, for any $1\le T \le x^{9/10}$, 
\begin{equation}
\label{xkeyidr2y} 
\sum_{n\le x} f(n) = \int_0^{\eta}\int_0^{\eta} \frac{1}{\pi i } \int_{c_0-iT}^{c_0+iT} 
F(s+\beta)  \cdot   \sum_{y<m<x/y} \frac{\Lambda_f(m)}{m^{s-\beta}} \cdot
\sum_{y<m<x/y} \frac{\Lambda_f(m)}{m^{s+\beta}} \cdot
\frac{x^{s-\alpha-\beta}}{s-\alpha-\beta}\ ds\ d\beta\ d\alpha
\end{equation} 
\[
   +O\Big(   \frac{x}{\log x}    + \frac{x }{T} \left( \frac{\log x} {\log y}\right)^{\kappa} \Big) .
\]
\end{proposition}

In fact, the proofs here simplify a little because in bounding some error terms there is no sum over smooth numbers. For instance, in the proof of Lemma \ref{errors23}
 we only have the terms with  $m=1$, so that there is no Euler product over primes $p\leq y$,  and hence no $(\log y)^\kappa$ term.    
  
We can obtain analogues of our previous results for multiplicative functions supported on large primes, 
by following the earlier proofs, using ${\mathcal S}(s)=1$ and replacing applications of Proposition 
\ref{xkeyidr1} with Proposition \ref{xkeyidr1y} above.  
Thus if $f\in {\mathcal C}(\kappa)$, with $f(n)=0$ if $n$ has a prime factor $\leq y$, 
where $(\log x)^{2(\kappa +2)}  \leq y \leq \sqrt{x}$, then the argument of Hal\'asz's Theorem \ref{GenHal} with $T = \log x \cdot  (\frac{\log x}{\log y})^{\kappa}$ yields 
$$ 
\sum_{n\le x} f(n)  \ll_{\kappa} \frac{x}{\log x} \Big( \int_{1/\log x}^{2/\log y} \Big( \max_{|t| \le (\frac{\log x}{\log y})^{\kappa} } \Big| \frac{F(1+\sigma+it)}{1+\sigma+it}\Big| 
  \Big) \frac{d\sigma}{\sigma} + 1\Big).  
 $$
We may now follow the proof of Corollary \ref{HalCor} using the bound (for $0 < \sigma \le 2/\log y$)
$$ 
\Big| \frac{F(1+\sigma+it)}{1+\sigma+it}\Big| \ll_{\kappa} \exp\Bigl\{\kappa \sum_{y < p \leq e^{1/\sigma}} \frac{1}{p^{1+\sigma}}\Bigr\} \leq \exp\Bigl\{\kappa \sum_{y < p \leq e^{1/\sigma}} \frac{1}{p}\Bigr\} \ll_{\kappa} \left(\frac{1}{\sigma \log y}\right)^{\kappa} $$
in place of the bound $| {F(1+\sigma+it)}/{(1+\sigma+it)} | \leq \zeta(1+\sigma)^{\kappa} \ll 1/\sigma^{\kappa}$.  
Thus we obtain that
\begin{equation}
\label{fullsieved}
\sum_{n\le x} f(n) \ll \frac{x}{\log x} \Big(  (1+M)e^{-M}    \Big(\frac{\log x}{\log y}\Big)^{\kappa} +  1\Big) ,
\end{equation}
where 
 $$
  \max_{|t| \le (\frac{\log x}{\log y})^{\kappa}  }  \Big| \frac{F(1+1/\log x+it)}{1+1/\log x+it}\Big| =: e^{-M} \Big(\frac{\log x}{\log y}\Big)^{\kappa} . 
  $$

Next we turn to the analogue of Theorem \ref{ShortHal} in this setting.   We obtain, assuming that $x^{2\delta} (\log x)^{2(\kappa+2)} \leq y \leq \sqrt{ x}$
  and $0<\delta<\frac 14$, and taking $T = x^{\delta} \log^{2}x \cdot  (\frac{\log x}{\log y})^{\kappa}$,
\begin{eqnarray}\label{shortsieved}
\sum_{x<n\le x+x^{1-\delta}} f(n)  &\ll & \frac{x^{1-\delta}}{\log x} \Big( \int_{1/\log x}^{2/\log y} \Big( \max_{|t| \le x^\delta (\frac{\log x}{\log y})^{\kappa} } \Big|  F(1+\sigma+it) \Big| 
  \Big) \frac{d\sigma}{\sigma} +
1 \Big) \nonumber \\
&\ll & \frac{x^{1-\delta}}{\log x} \Big(   (1+M_\delta)e^{-M_\delta}\Big(\frac{\log x}{\log y}\Big)^{\kappa} +
1  \Big)
\end{eqnarray}  
where
$$
 \max_{|t| \le x^\delta (\frac{\log x}{\log y})^{\kappa} } \Big|  F(1+1/\log x+it) \Big| =: e^{-M_\delta} \Big(\frac{\log x}{\log y}\Big)^{\kappa} .
$$

Similarly, we may formulate an analogue for arithmetic progressions.   Thus, if $(a,q)=1$ and 
 $q^{4}(\log x)^{2(\kappa +2)} \leq y \leq \sqrt{x}$, then (for $f \in {\mathcal C}(\kappa)$ with $f(p^k)=0$ for $p\le y$) 
\begin{align}\label{apsieved} 
\sum_{\substack{n\leq x\\ n\equiv a \pmod q}}  f(n) &\ll  \frac{1}{\phi(q)} \frac{x}{\log x} \Big( \int_{1/\log x}^{2/\log y} \Big( \max_{\substack{ \chi \pmod q \\ |t| \le (\frac{\log x}{\log y})^{\kappa} }} \Big| \frac{F_{\chi}(1+\sigma+it)}{1+\sigma+it} \Big| 
  \Big) \frac{d\sigma}{\sigma} 
  +1 \Big) \nonumber \\
&\ll  \frac{1}{\phi(q)} \frac{x}{\log x} \Big( (1+M_q)e^{-M_q}\Big(\frac{\log x}{\log y}\Big)^{\kappa} 
+1  \Big), 
\end{align}  
where
$$
 \max_{\substack{\chi \pmod q \\ |t| \le (\frac{\log x}{\log y})^{\kappa}}} \Big|  \frac{F_{\chi}(1+1/\log x+it)}{1+1/\log x + it} \Big| =: e^{-M_q} \Big(\frac{\log x}{\log y}\Big)^{\kappa} .
$$

For our work on primes in arithmetic progressions what will be useful are analogues of 
Theorem \ref{thm1.8}.  Using Proposition \ref{Rep2}, we find that if $q\log x \leq y \leq \sqrt{x}$ then for any natural number $J$ there are at most $J$ 
characters $\chi \pmod q$ with 
$$ 
\max_{|t| \le (\frac{\log x}{\log y})^{\kappa} }  \Big| F_{\overline{\chi}} (1+1/\log x +it) \Big|  
\asymp \max_{|t| \le (\frac{\log x}{\log y})^{\kappa} }  |{\mathcal L}_{\overline{\chi}}(1+1/\log x+it)| 
\gg_{\kappa, J} \Big( \frac{\log x}{\log y}\Big)^{\frac{\kappa}{\sqrt{J+1}}} . 
$$ 
If ${\mathcal X}_J$ denotes the set of at most $J$ characters for which the above holds, then for all $\psi \not\in {\mathcal X}_J$ by \eqref{fullsieved} we have 
\begin{equation} 
\label{7.6} 
|S_f(x,\psi)| \ll_{\kappa,J} \frac{x}{\log x} \Big( \frac{\log x}{\log y}\Big)^{\frac{\kappa}{\sqrt{J+1}}} \log \Big( \frac{\log x}{\log y} \Big).
\end{equation} 
Moreover, arguing as in Section 6, we find that 
\begin{equation} 
\label{7.7} 
\Big| \sum_{\substack{n\le x \\ n\equiv a \pmod q}} f(n) - \frac{1}{\phi(q)} \sum_{\chi \in {\mathcal X}_J} \chi(a) S_f(x,\chi) \Big| 
\ll_{\kappa,J} \frac{x}{\phi(q) \log x} \Big(\frac{\log x}{\log y} \Big)^{\frac{\kappa}{\sqrt{J+1}}} \log \Big(\frac{\log x}{\log y} \Big). 
\end{equation}

\subsection{Primes in short intervals:  Proof of Corollary \ref{Hoh}}
Let $\delta > 0$ be small, and set 
$$
y = \max\{(x^{\delta}\log^{3}x)^{2}, \exp\{(\log x)^{3/4}\}\}.
$$  
Let 
$1_y$ be the indicator function of those numbers with all prime factors $> y$, and $\mu_y$ be the M\"{o}bius function restricted to numbers with all prime factors $> y$.
Our interest is in counting primes in the interval $[x,x+x^{1-\delta}]$ and for $n$ in this interval we 
may write $\Lambda_{1_y}(n)$, which is usually $\Lambda(n)$, as 
$$ 
 \Lambda_{1_y}(n)  = \sum_{k\ell = n} \mu_y(k) 1_y(\ell) \log \ell = \Big(\sum_{\substack{k\ell =n \\ k\le K}} + \sum_{\substack{k\ell =n \\ \ell \le L \\ k>K}} \Big) 
\mu_y(k) 1_y(\ell) \log \ell,  
$$ 
where $K$ and $L$ are parameters to be chosen later with $KL=x+x^{1-\delta}$, $y^2 \le K, L \le x/y^2$ and $L\le \sqrt{x}$.  
Thus 
\begin{eqnarray}
\label{7.8} 
\sum_{x<n\leq x+x^{1-\delta}} \Lambda(n) &+& O(y) =  \sum_{x<n\leq x+x^{1-\delta}} \Lambda_{1_y}(n)  \nonumber \\
& = & \sum_{k\leq K} \mu_y(k) \sum_{\frac{x}{k}<\ell\leq \frac{x+x^{1-\delta}}{k}}   1_y(\ell)\log \ell + \sum_{\ell\leq L}  1_y(\ell)\log \ell  \sum_{\substack{\frac{x}{\ell}<k\leq \frac{x+x^{1-\delta}}{\ell} \\ k > K}} \mu_y(k) . 
\end{eqnarray}

We begin by analyzing the first term in \eqref{7.8}.   Temporarily setting $u= (\log L)/(2\log y)$, a weak form of the fundamental lemma 
from sieve theory (see for example Corollary 6.10 of \cite{FI}) gives 
$$ 
\sum_{x/k \le \ell \le (x+x^{1-\delta})/k}  1_y(\ell) = \frac{x^{1-\delta}}{k} \prod_{p\le y} \Big(1- \frac 1p \Big) (1 + O(e^{-u})),  
$$  
and so 
$$ 
\sum_{x/k \le \ell \le (x+x^{1-\delta})/k}  1_y(\ell) \log \ell = \frac{x^{1-\delta}}{k} \Big( \log \frac xk  + O(x^{-\delta}) \Big) \prod_{p\le y} \Big(1- \frac 1p \Big) ( 1+ O(e^{-u}) ). 
$$ 
Thus the first term in \eqref{7.8} is 
$$ 
x^{1-\delta} \prod_{p\le y} \Big(1- \frac 1p\Big) \sum_{k \le K} \frac{\mu_y(k)}{k} \log \frac{x}{k} + O\Big( \Big( x^{1-\delta} e^{-u}\frac{\log x}{\log y}  + \frac{x^{1-2\delta}}{\log y} \Big) \sum_{k\le K} 
\frac{|\mu_y(k)|}{k}\Big). 
$$ 
Using the sieve again, the above equals 
\begin{equation} 
\label{7.11} 
x^{1-\delta} \prod_{p\le y} \Big(1- \frac 1p\Big) \sum_{k \le K} \frac{\mu_y(k)}{k} \log \frac{x}{k} + O \Big( x^{1-\delta } e^{-u} \Big(\frac{\log x}{\log y}\Big)^2 + 
x^{1-2\delta} \frac{\log x}{(\log y)^2}\Big). 
\end{equation}  


Setting this aside for the moment, we turn to the second term in \eqref{7.8}.  Applying \eqref{shortsieved} to $f=\mu_y$ (and with $\kappa=1)$ we 
obtain 
$$
\sum_{\substack{x/\ell \le k \le (x+x^{1-\delta})/\ell \\ k> K } } \mu_{y}(k) \ll \frac{x^{1-\delta}}{\ell \log (x/\ell)} \Big((1+M_{\delta,\ell}) e^{-M_{\delta,\ell}} \frac{\log(x/\ell)}{\log y} + 1\Big) 
$$ 
where 
$$
 \max_{|t| \le x^\delta  \frac{\log (x/\ell)}{\log y}  }\Big| \prod_{p > y} \Big(1 - \frac{1}{p^{1+1/\log (x/\ell) + it}} \Big) \Big| =: e^{-M_{\delta,\ell}}  \frac{\log (x/\ell)}{\log y}  .
$$ 
To estimate the above quantity, we invoke the following proposition.  



\begin{proposition}\label{muyprop}
Let $x \geq y \geq 2$, and suppose $t \in \mathbb{R}$. Then
$$ 
\Big| \prod_{p > y} \Big(1 - \frac{1}{p^{1+1/\log x + it}} \Big) \Big| \ll \Big(\frac{\log x}{\log y}\Big)^{3/4} \Big(\frac{\log(y + |t|)}{\log y} \Big)^{1/4} . 
$$
\end{proposition}

In our application of Proposition \ref{muyprop} we have $|t|\leq x^\delta  \frac{\log (x/\ell)}{\log y} \leq
x^\delta   \log x \leq y^{1/2}$ so that $\log(y + |t|) \asymp \log y$. Therefore
 Proposition \ref{muyprop}  implies that $e^{-M_{\delta,\ell}} \ll ((\log y)/\log(x/\ell))^{1/4}$,   so that  
\begin{eqnarray*}
 \sum_{\substack{x/\ell \le k \le (x+x^{1-\delta})/\ell \\ k> K } } \mu_{y}(k) & \ll & \frac{x^{1-\delta}}{\ell \log y} \Big( \frac{\log y}{\log (x/\ell)}\Big)^{\frac 14} \log 
 \Big( \frac{\log (x/\ell)}{\log y} \Big). 
 \end{eqnarray*} 
 Performing now the sum over $\ell$, and using the sieve once again, the second term in \eqref{7.8} is
  \begin{equation} 
  \label{7.12}  
  \ll x^{1-\delta} \Big( \frac{\log L}{\log y}\Big)^2 \Big( \frac{\log y}{\log (x/L)}\Big)^{\frac 14} \log \Big(\frac{\log (x/L)}{\log y}\Big). 
 \end{equation}


We now choose $u = 3\log (\frac{\log x}{\log y})$ so that $e^{-u}=( \frac{\log y}{\log x})^3$, and 
 $L=y^{2u} \leq \sqrt{x}$.
Then, combining \eqref{7.11} and \eqref{7.12}, we conclude that  
\begin{eqnarray}
\label{7.13} 
\sum_{x<n\leq x+x^{1-\delta}} \Lambda(n)  
& = & x^{1-\delta} \prod_{p\leq y} \Big( 1 -\frac 1p \Big)  \sum_{k\leq K} \frac{ \mu_y(k)}{k} \log \frac xk  
+ O\Big(x^{1-\delta} \Big(\frac{\log y}{\log x} \Big)^{\frac 14} \Big(\log \Big(\frac{\log x}{\log y}\Big) \Big)^3 \Big) \nonumber \\
& = & x^{1-\delta}\prod_{p\leq y} \Big( 1 -\frac 1p \Big) \sum_{k\leq K} \frac{\mu_y(k)}{k} \log \frac xk  
 + O\Big(x^{1-\delta}\Big(\delta^{\frac 15} + \frac{1}{(\log x)^{\frac{1}{20}}} \Big) \Big) . 
 \end{eqnarray}
 
 It remains to simplify the main term in \eqref{7.13}.   We may finesse this issue as follows.  
Summing up  \eqref{7.13} over all intervals of length $x^{1-\delta}$ between $x$ and $2x$ we find 
$$ 
  \prod_{p\leq y} \Big( 1 -\frac 1p \Big) \sum_{k\leq K} \frac{\mu_y(k)}{k} \log \frac xk
   = \frac{1}{x} \sum_{x<n\leq 2x} \Lambda(n) + O\Big( \delta^{\frac 15}+ \frac{1}{(\log x)^{\frac 1{20}}} \Big).
 $$
The prime number theorem (which, incidentally, may be deduced from Hal\'{a}sz's theorem) now allows us to evaluate the left side above.  
We conclude that 
\begin{equation}\label{Hoheisel}
\sum_{x<n\leq x+x^{1-\delta}} \Lambda(n)  = x^{1-\delta} \Big(1+O\Big(\delta^{\frac 15} + \frac{1}{(\log x)^{\frac{1}{20}}}\Big) \Big),
\end{equation}
which establishes Corollary \ref{Hoh}.   It remains lastly to prove Proposition \ref{muyprop}.  

\begin{proof}[Proof of Proposition \ref{muyprop}] Set $c_0 = 1+1/\log x$. If $|t| \le 1/\log y$ then the product in the proposition may be 
estimated trivially as 
$$ 
\le \frac{1}{|\zeta(c_0 +it)|} \prod_{p\le y} \Big(1- \frac 1p\Big)^{-1} \ll  \frac{\log y}{\log x}\ll 1. 
$$ 
We may therefore assume that $|t|\ge 1/\log y$.  The elementary inequality $3+4\cos \theta +\cos 2\theta \ge 0$ yields (as in 
usual proofs of the zero-free region)  
 $$ 
 \Big| \prod_{p > y} \Big(1 - \frac{1}{p^{c_0 + it}} \Big) \Big|^{4} 
 \le \Big | \prod_{p > y} \Big(1 - \frac{1}{p^{c_0}} \Big)^{-3} \Big| \Big| \prod_{p > y} \Big(1 - \frac{1}{p^{c_0 + 2it}} \Big)^{-1} \Big| 
 \ll \Big(\frac{\log x}{\log y}\Big)^3    \Big|\sum_{n=1}^{\infty} \frac{1_{y}(n)}{n^{c_0+2it}} \Big|. 
 $$

Now Lemma 2.4 of Koukoulopoulos~\cite{Koukoulopoulos} implies that\footnote{Actually 
 Lemma 2.4 of \cite{Koukoulopoulos}  
  provides a much better dependence on $|t|$, as it uses the Vinogradov--Korobov estimate for sums $\sum_{n} n^{it}$. This will improve $ \log(y + |t|)$ to  $\log y +(\log 3+ |t|)^{2/3+o(1)} $ in the statement of Proposition \ref{muyprop}, but it has no effect on the strength of \eqref{Hoheisel}.}  for 
 all $z\ge y\ge 2$
$$ 
\sum_{n \leq z} 1_{y}(n) n^{-2it} = \frac{z^{1-2it}}{1-2it} \prod_{p \leq y} \Big(1-\frac{1}{p}\Big) + O\Big(\frac{  z^{1-c/\log(y+|t|)}}{\log y}\Big)  . 
$$
Partial summation gives 
\[
\sum_{n=1}^{\infty} \frac{1_{y}(n)}{n^{c_0+2it}} 
= 1 + \int_{y}^{\infty} \frac{1}{z^{c_0}} d\Big(\sum_{n \leq z} 1_{y}(n) n^{-2it}\Big) 
\]
and we now input the asymptotic formula above.  This leads to the main term 
\[
1 + \int_{y}^{\infty} \frac{1}{z^{c_0}} d\Big(\frac{z^{1-2it}}{1-2it} \prod_{p \leq y} \Big(1-\frac{1}{p}\Big)\Big)    
= 1+\int_{y}^{\infty} \frac{1}{z^{c_0 + 2it}} \prod_{p \leq y}\left (1-\frac{1}{p}\right) dz \ll 1+�\frac{1}{|t|\log y} \ll 1,
\]
as $|t| \geq 1/\log y$. Moreover the error term contributes (after integrating by parts) 
$$ 
\ll \frac 1{\log y} +\int_{y}^{\infty} \frac{1}{z^{c_0+1} }  \frac{z^{1-c/\log(y+|t|)}}{\log y} dz  \ll \frac{\log(y + |t|)}{\log y}, 
$$
which completes the proof.

Alternatively, we may upper bound $\prod_{p> y}|1 - 1/p^{c_0+2it}|^{-1}$ by using the monotonicity principle described  
in Lemma \ref{lem5.2}.   If $|t| \le 1/\log y$, or if $y+|t| \ge x$ then Proposition \ref{muyprop} is immediate, so
we may assume that $|t| \ge 1/\log y$ and $y+|t| \le x$.  Next observe that  
\begin{align*}
\prod_{p> y} \Big| 1-\frac{1}{p^{c_0 +2it}} \Big|^{-1} &\ll \prod_{y< p\le (y+|t|)} \Big(1 -\frac 1p\Big)^{-1} 
\prod_{p> (y+|t|)} \Big| 1-\frac{1}{p^{c_0 +2it}}\Big|^{-1} \\
&\ll \Big( \frac{\log (y+|t|)}{\log y}\Big) \frac{|\zeta(c_0 + 2it)|}{|\zeta(1+1/\log (y+|t|) + 2it)|}. 
\end{align*}
Next in the range $|t| \ge 1/\log y$, using Stirling's formula we obtain  
$$ 
\frac{|\zeta(c_0+2it)|}{|\zeta(1+1/\log (y+|t|) +2it)|} 
\asymp \frac{|\xi (c_0 +2it)|}{|\xi(1+1/\log (y+|t|) +2it)|} \le 1, 
$$
where $\xi(s) = s(s-1) \pi^{-s/2} \Gamma(s/2) \zeta(s)$ denotes Riemann's $\xi$-function, and 
the final inequality holds because $y+|t| \le x$ and $|\xi(\sigma+ 2it)|$ is monotone increasing for $\sigma >1$.   
%
\end{proof}


\subsection{Primes in progressions:  Theorem \ref{thm1.11} and Corollaries \ref{cor1.12} and \ref{cor1.13}}\label{linniksec}

Set $y= q^4 (\log x)^6$, and let $1_y$ and $\mu_y$ be as in the previous section. Recall that $Q=q\log x$, so that $\log y \asymp \log Q$.   
Let ${\mathcal X}_J$  be the set of non-principal characters $\chi \pmod q$ for which 
$$ 
\max_{|t| \le \log x}  \prod_{Q \le p \le x} \Big| 1- \frac{\chi(p)}{p^{1+it}}\Big| 
\asymp \max_{|t| \le \log x} \prod_{ p>y } \Big| 1- \frac{\chi(p)}{p^{1+1/\log x + it}} \Big| \gg_J \Big( \frac{\log x}{\log y} \Big)^{ \frac{1}{\sqrt{J+1}} }.  
$$ 
By Proposition \ref{Rep2} applied to the function $\mu_y$, we know that ${\mathcal X}_J$ contains no more than $J$ elements.
Let $\widetilde{\mathcal X}_J$ denote 
the set ${\mathcal X}_J$ extended to include the principal character.

As before, let $K$ and $L$ be parameters with $y^2 \le K, L \le x/y^2$, and 
with $KL= x$ and $L\le \sqrt{x}$.   Using the decomposition of $\Lambda_{1_y}(n)$, as in \eqref{7.8}, we obtain that 
$$ 
\sum_{\substack{ n\le x \\ n\equiv a \pmod q}} \Lambda(n ) -\frac{1}{\phi(q)} \sum_{\chi \in \widetilde{\mathcal X}_J} \overline{\chi}(a) 
\sum_{n\le x} \Lambda(n) \chi(n )  + O\Big( y \frac{\log x}{\log y} \Big)
$$ 
equals 
\begin{align} 
\label{7.3.1} 
\sum_{\substack{ k\le K \\ (k,q)=1} }& \mu_y(k) \Big( \sum_{\substack{\ell \le x/k \\ \ell \equiv a/k \pmod q} } 1_y(\ell) \log \ell 
- \frac{1}{\phi(q) } \sum_{\chi \in \widetilde{\mathcal X}_J} \overline{\chi}(a/k) \sum_{\ell \le x/k} 1_y(\ell) \chi(\ell) \log \ell \Big) \nonumber \\ 
&+ \sum_{\substack{ \ell \le L \\ (\ell,q)=1 }} 1_y(\ell)\log \ell \Big( \sum_{\substack{ K< k \le x/\ell\\  k\equiv a/\ell \pmod q}} \mu_y(k) 
- \frac{1}{\phi(q)} \sum_{\chi \in \widetilde{\mathcal X}_J} \overline{\chi}(a/\ell) \sum_{K < k \le x/\ell} \mu_y(k) \chi(k) \Big). 
\end{align} 

We now follow the argument in the previous section to handle the two sums above, starting with the first sum.  Set $u= (\log L)/(2\log y)$, 
as in the previous section.  Using partial summation and the fundamental lemma from sieve theory (Corollary 6.10 of \cite{FI}) we 
obtain (for any reduced residue class $b\pmod q$) 
\begin{equation} \label{FIconseq}
 \sum_{\substack{\ell \leq x/k \\ \ell \equiv b \pmod q}} 1_y(\ell)\log \ell =
 \frac{x}{k q} \prod_{\substack{p \leq y \\ p \nmid q}} \Big(1-\frac{1}{p}\Big)\log \frac{x}{k e} \Big( 1+ O(e^{-u}) \Big).  
  \end{equation}
If $\chi \in {\mathcal X}_J$ is a non-principal character, then using \eqref{FIconseq} 
\[
 \sum_{ \ell \leq x/k  }  1_y(\ell)\chi(\ell) \log \ell =
  \sum_{b \pmod q}  \chi(b) \sum_{\substack{\ell \leq x/k , \\ \ell \equiv b \pmod q}} 1_y(\ell)\log \ell  = 
  O\Big( e^{-u} \frac{x\phi(q)}{kq} \prod_{\substack{p\le y\\ p\nmid q}} \Big(1-\frac 1p\Big) \log\frac{x}{ke}\Big). 
 \]
 For the principal character, we obtain similarly that 
 $$ 
 \frac{1}{\phi(q)} \sum_{ \ell \le x/k} 1_y(\ell) \chi_0(\ell) \log \ell =  \frac{x}{k q} \prod_{\substack{p \leq y \\ p \nmid q}} \Big(1-\frac{1}{p}\Big)\log \frac{x}{k e} \Big( 1+ O(e^{-u}) \Big).
 $$ 
 Putting these results together for the inner sums over $\ell$ in the first term of \eqref{7.3.1}, we may bound that term by 
 \begin{equation} 
 \label{7.3.2} 
\ll_J \sum_{\substack{ k \le K }} |\mu_y(k)|  \Big( e^{-u} \frac{x}{kq} \prod_{\substack{p\le y\\ p\nmid q}} \Big(1-\frac 1p\Big) \log\frac{x}{ke}\Big) \ll 
e^{-u} \frac{x}{\phi(q)} \Big( \frac{\log x}{\log y}\Big)^2.
\end{equation}

Now we turn to the second term in \eqref{7.3.1}.  By \eqref{7.7} (with $x$ there replaced by $x/\ell$ and then $K$ here), we obtain 
\begin{align*}
\Big| \sum_{\substack{K < k\le x/\ell \\ k\equiv a/\ell \pmod q}} \mu_y(k) 
&- \frac{1}{\phi(q)} \sum_{\chi \in \widetilde{\mathcal X}_J} \overline{\chi}(a/\ell) \sum_{K < k \le x/\ell} 
\mu_y( k) \chi(k) \Big| 
\\
&\ll_{J} \frac{x}{\ell \phi(q) \log x} \Big( \frac{\log x}{\log y}\Big)^{\frac{1}{\sqrt{J+1}} } \log \Big( \frac{\log x}{\log y}\Big). 
\end{align*} 
%
%
Therefore the second term in \eqref{7.3.1} may be bounded by 
\begin{align*}
&\ll \frac{x}{ \phi(q) \log x} \Big( \frac{\log x}{\log y}\Big)^{\frac{1}{\sqrt{J+1}}} \log \Big( \frac{\log x}{\log y} \Big)  \sum_{\substack{\ell \le L \\ (\ell,q)=1}} \frac{1_y(\ell)}{\ell} \log \ell  \\
&\ll \frac{x}{\phi(q)} \Big(\frac{\log L}{\log y}\Big)^2 \Big( \frac{\log y}{\log x}\Big)^{1-\frac{1}{\sqrt{J+1}}} \log \Big( \frac{\log x}{\log y} \Big),
\end{align*} 
upon using the sieve once again.
Recall that $L= y^{2u}$ and choose once again $u = 3\log(\frac{\log x}{\log y})$, (so that in particular $L \leq \sqrt{x}$, provided $\frac{\log x}{\log y}$ is large). Then combining 
the above with \eqref{7.3.1} and \eqref{7.3.2} we conclude that 
\begin{equation*}
\Big| 
\sum_{\substack{n\leq x\\ n\equiv a \pmod q}} \Lambda(n) - \frac{1}{\phi(q)} \sum_{\chi \in \widetilde{\mathcal X}_J} 
\overline{\chi}(a)  \sum_{n\le x} \Lambda(n) \chi(n) \Big| 
\ll   \frac{x}{\phi(q)} \Big(\frac{\log y}{\log x} \Big)^{1-\frac{1}{\sqrt{J+1}}} \Big(\log \Big(\frac{\log x}{\log y}\Big) \Big)^3. \nonumber 
\end{equation*}
This establishes \eqref{1.5} of Theorem \ref{thm1.11}, using the prime number theorem to account for the contribution of the principal character.

Now if $\psi$ is a non-principal character mod $q$, with $\psi\not\in {\mathcal X}_J$ then
\begin{align*}
\sum_{\substack{n\leq x}} \Lambda(n) \psi(n) &= \sum_{a} \psi(a) \Big( 
\sum_{\substack{n\leq x\\ n\equiv a \pmod q}} \Lambda(n) - \frac{1}{\phi(q)} \sum_{\chi \in \widetilde{\mathcal X}_J} 
\overline{\chi}(a)  \sum_{n\le x} \Lambda(n) \chi(n) \Big) \\
&\ll x  \Big(\frac{\log y}{\log x} \Big)^{1-\frac{1}{\sqrt{J+1}}} \Big(\log \Big(\frac{\log x}{\log y}\Big) \Big)^3,
\end{align*} 
by the last displayed equation. This implies  \eqref{1.4},  completing the proof of Theorem \ref{thm1.11}. 
Alternatively, arguing as in our proof of \eqref{1.5} (with the same parameters, but replacing the use of \eqref{7.7} by \eqref{fullsieved} 
applied to $f=\mu_y\chi$), we obtain that for any non-principal character $\chi \pmod q$ 
 \begin{equation} 
 \label{7.16} 
 \sum_{n\le x} \Lambda(n) \chi(n) 
 \ll  x\frac{\log y}{\log x} \Big( \max_{|t| \le \frac{\log x}{\log y}} \prod_{y< p \le x} \Big| 1- \frac{\chi(p)}{p^{1+it}} \Big| \Big) 
 \Big( \log \Big( \frac{\log x}{\log y} \Big) \Big)^3.
 \end{equation} 
 This again proves \eqref{1.4}, and moreover \eqref{7.16} will be used in the proof of Corollary \ref{cor1.12}.  
\hfill \qed
\medskip

 \begin{proof}[Proof of Corollary \ref{cor1.13}]    Corollary \ref{cor1.13} follows from Theorem \ref{thm1.11} in much the same way that Theorem \ref{PLSthm} 
 was deduced from Theorem \ref{1.8}, in section \ref{secapsandpls}.  
\end{proof}



 \begin{proof}[Proof of Corollary \ref{cor1.12}]  Apply our work above with $J=1$.   If ${\mathcal X}_1$ is empty, then there is no exceptional 
 character and a stronger form of \eqref{1.6} follows from Theorem \ref{thm1.11}.   If ${\mathcal X}_1$ is non-empty, then the unique character 
 that it contains must be real (else its conjugate would also belong to the set).  Now \eqref{1.8} follows at once from Theorem \ref{thm1.11}.  
 Using \eqref{7.16} and \eqref{1.8} we see that if \eqref{1.6} fails, then we must have 
 $$ 
  \max_{|t| \le \frac{\log x}{\log y}} \prod_{y< p \le x} \Big| 1- \frac{\chi(p)}{p^{1+it}} \Big| \gg \frac{\log x}{\log y} \Big(\log\Big( \frac{\log x}{\log y}\Big) \Big)^{-4}. 
 $$
 If the maximum above occurs at a point $t_0$, then taking logarithms we have 
 \begin{equation} 
 \label{7.17}
 \sum_{Q \le p \le x} \frac{1+\chi(p) \cos(t_0 \log p)}{p} \le 4 \log \log\Big( \frac{\log x}{\log y}\Big) + O(1). 
 \end{equation} 
 
 It remains to deduce \eqref{1.7} from \eqref{7.17}.  From \eqref{7.17} it follows that 
 $$ 
 \sum_{Q \le p \le x} \frac{1- |\cos (t_0 \log p)|}{p} \le 4 \log \log \Big( \frac{\log x}{\log y}\Big)+ O(1).
 $$ 
 Using the prime number theorem and partial summation (as in our proof of Lemma \ref{lem4.1}, and recalling that $\int_0^1 |\cos (2\pi u)| du = 2/\pi$)
 we find that for $|t_0 |\le (\log x)/\log y$,
\begin{eqnarray}
 \sum_{Q \le p\le x} \frac{1-|\cos(t_0 \log p)|}{p} & \geq & \sum_{\max\{Q, e^{1/|t_0|}\} \le p\le x} \frac{1-|\cos(t_0 \log p)|}{p}                          \nonumber \\
& \ge & \Big(1 -\frac{2}{\pi} + o(1)\Big) \log \min \Big( |t_0| \log x, \frac{\log x}{\log y}\Big). \nonumber
\end{eqnarray}
Comparing the above two estimates, provided $\frac{\log x}{\log y}$ is large enough we obtain (since $\pi/(\pi -2) <3$) 
$$ 
|t_0| \log x \ll \Big( \log \Big(\frac{\log x}{\log y}\Big) \Big)^{12}. 
$$ 
Hence 
$$ 
\sum_{Q\le p\le x} \frac{1+\chi(p)}{p} 
\le \sum_{Q\le p\le x} \Big( \frac{1+\chi(p)\cos(t_0\log p)}{p} + \frac{|1-\cos (t_0 \log p)|}{p} \Big). 
$$ 
Bounding the first sum using \eqref{7.17} and the second sum by 
$$ 
O(1) + \sum_{e^{1/|t_0|}\le p\le x} \frac{2}{p} \le 2 \log (|t_0| \log x)+O(1)
\le 24 \log \log \Big( \frac{\log x}{\log y} \Big)+ O(1), 
$$
we deduce \eqref{1.7}. 
 \end{proof}

%


\section{Linnik's Theorem:  Proof of Corollary \ref{Linnik} }

\noindent Let $L$ be a sufficiently large fixed constant; what ``sufficiently large'' means will evolve from satisfying certain inequalities in the proof. Suppose  that $(a,q)=1$ and every  prime  $p\equiv a\pmod q$ is 
$>q^L$.   We appeal to Corollary \ref{cor1.12} with $x=q^L$, where $L$ is chosen large enough that the error terms in 
in Corollary \ref{cor1.12} are all much smaller than the main terms.

Clearly \eqref{1.6} does not hold, and therefore we must have an exceptional real character $\chi \pmod q$ for which \eqref{1.7} and 
\eqref{1.8} hold.  In particular, \eqref{1.7} gives 
$$ 
\sum_{q \le p \le q^L} \frac{1+\chi(p)}{p} \le 30 \log \log L + O(1).  
$$ 

Now put $L_0 = \exp(\sqrt{\log L})$, and apply Theorem \ref{thm1.11} with $J=1$ for $x$ in the 
range $q^{L_0} \le x \le q^L$.   The set ${\mathcal X}_1$ consists precisely of one character, namely the 
exceptional character $\chi$: indeed 
\begin{align*}
\max_{|t| \le \frac{\log x}{\log Q}} \prod_{Q\le p\le x} \Big| 1- \frac{\chi(p)}{p^{1+it}} \Big| 
&\ge \prod_{Q \le p \le x} \Big| 1- \frac{\chi(p)}{p}\Big| \gg \frac{\log x}{\log Q} \exp\Big(- \sum_{Q\le p\le x} \frac{1+\chi(p)}{p} \Big) \\
&\gg \frac{\log x}{\log Q} (\log L)^{-30}.
\end{align*}
Theorem \ref{thm1.11}, equation  \eqref{1.5}, now tells us that (since there are no primes $p\equiv a \pmod q$ 
below $q^L$)  for all $q^{L_0} \le x\le q^{L}$ we have
$$ 
-\chi(a) \sum_{n\le x} \Lambda(n) \chi(n) = x + O\Big( x \Big(\frac{\log q}{\log x}\Big)^{\frac 14}\Big). 
$$ 
Partial summation, using this relation, yields 
$$ 
\sum_{q^{L_0} \le p \le q^{L}} \frac{\chi(p)}{p} = -\chi(a) \log \frac{L}{L_0} + O(L_0^{-\frac 14}), 
$$ 
so that 
$$ 
\sum_{q^{L_0} \le p \le q^{L}} \frac{1+\chi(p)}{p} = (1-\chi(a)) \log \frac{L}{L_0} + O(L_0^{-\frac 14}). 
$$ 
We  proved above that the left hand side is $ \le 30 \log \log L + O(1)$, and so  $\chi(a)$ must be $1$, implying that
\begin{equation} 
\label{8.1} 
\sum_{q^{L_0} \le p \le q^{L}} \frac{1+\chi(p)}{p} =  O(L_0^{-\frac 14}). 
\end{equation}

At this stage, we invoke a sieve estimate from Friedlander and Iwaniec \cite{FI} (see Proposition 24.1, and the preceding discussion).   Put 
$$ 
a(n) = \sum_{d| n} \chi(d) = \prod_{p^a ||n}(1 + \chi(p) + ... + \chi(p)^a) .  
$$
This is a natural object in number theory:  
If $\chi$ is the quadratic character that corresponds to a fundamental discriminant $D$, then $a(n)$ counts (after accounting for automorphisms) 
the number of representations of $n$ by binary quadratic forms of discriminant $D$.  Friedlander and Iwaniec 
consider the sequence $a(n)$ over $n\le q^L$ and $n\equiv a\pmod q$ (with $\chi(a)=1$), and sieve this sequence 
to extract the primes.  (Note that if $\chi(a)=1$ then for a prime $p \equiv a \pmod q$ one has $a(p) = 1+\chi(p) =2$.)   They 
establish that if $x \ge q^{O(1)}$, and $z \le x^{1/8}$ then 
\begin{equation} 
\label{8.2} 
\pi(x;q,a) = L(1, \chi) \frac{x}{q} \prod_{\substack{p\le z \\ p\nmid q}} \Big(1-\frac 1p \Big) \Big(1- \frac{\chi(p)}{p} \Big) 
\Big( 1 + O\Big( \frac{\log z}{\log x} + \frac{1}{z} + \sum_{\substack{ z<p \le x\\ \chi(p) =1 }} \frac 1p\Big)\Big). 
\end{equation} 
The idea behind this result is that one can usually sieve (as in the fundamental lemma) up to $z$ being 
a small power of $x$, and if there are few large primes $p$ for which $\chi(p) =1$ then one has a sifting 
problem of low dimension so that it becomes possible to estimate accurately the number of primes left after sieving.   The estimate \eqref{8.2}  is 
useful only if $\sum_{\substack{z<p\le x\\ \chi(p)= 1}} 1/p$ is very small, but thanks to \eqref{8.1} this is precisely the situation we are in!

We therefore apply \eqref{8.2} taking $x=q^L$ and $z=q^{L_0}$, so that
$$
\pi(x;q,a) = L(1,\chi) \frac{x}{q} \prod_{\substack{p\le q^{L_0} \\ p\nmid q}}\Big(1-\frac 1p \Big) \Big(1-\frac{\chi(p)}{p}\Big) 
\Big( 1+ O\Big(  {L_0^{-\frac 14}}\Big)\Big). 
$$ 
If $L$ is suitably large, it then follows that 
$$ 
\pi(x;q,a) \ge \frac{1}{2} L(1,\chi) \frac{x}{q} \prod_{p\le q^{L_0}} \Big(1 -\frac{1}{p}\Big)^2, 
$$ 
say, and we have shown that the least prime $p\equiv a \pmod q$ is below $q^L$.   Note that we do not need Siegel's theorem 
here; an effective lower bound of the form $L(1,\chi) \gg 1/\sqrt{q}$ suffices.


\section{Examples of Asymptotics} \label{sec: examples}

\noindent We will prove \eqref{NextPaper} in \cite{II}, and that will allow us to find asymptotics for the mean values of certain interesting $f$. Here we discuss three important examples given in \cite{II}.
\medskip

\noindent \textsl{Example 1}:  If $F(s)=\zeta(s)^\beta$ so that $f(n)=d_\beta(n)$, then
$$ 
\sum_{n\le x} f(n) =   \Big\{ \frac 1{\Gamma(\beta)}  +o(1)\Big\} x (\log x)^{\beta-1} .
$$
This is a well known result of Selberg and Delange.
\medskip


\noindent \textsl{Example 2}: \ \underline{$f(p)=g(\log p)$ where $g(t)=$ sign$(\cos (2\pi t))$}. Here $g(t)=1$ if $-\frac 14\leq t\leq \frac 14$ mod 1, and $g(t)=-1$ otherwise.  The Fourier coefficients of $g$ are given by 
\[
\widehat g(m)=  \frac{2}{ \pi m}  \cdot 
\begin{cases} 1 &\text{if} \ m\equiv 1 \pmod 4\\
-1&\text{if} \ m\equiv -1 \pmod 4\\
0 &\text{if} \ m\ \text{otherwise,}
\end{cases}
\]
and one can prove that
$F(c_0+2i\pi m) \sim c_m  (\log x)^{\widehat g(m)}$ for certain non-zero constants $c_m$, where $c_{-m}=\overline{c_m}$. Thus $t_1$ can be taken to be either $2\pi$ or $-2\pi$, and one can show that 
$$ 
\sum_{n\le x} f(n) = \left(  \text{Re}( c  x^{2i\pi } ) +   o(1)\right) \ x \  (\log x)^{2/\pi-1}  
$$ 
for some constant  $c\ne 0$. Taking $y=e^{1/4}$ with, say $t_1=-2\pi$, we deduce that
\begin{equation} \label{repel example}
\frac 1{x^{1+ it_1 }} \sum_{n\le x} f(n) - \frac 1{(x/y)^{1+ it_1 }} \sum_{n\le x/y} f(n) 
\sim    c  x^{4i\pi } \cdot  \left(\frac{4\log y}{\log x}\right)^{1-2/\pi}  .
\end{equation}
Hence the Lipschitz exponent for $f$ is at most $1-2/\pi$.
\smallskip

For any given $y$ in the range $1+\epsilon\leq y\leq \sqrt{x}$ we can let $f(p)=g(T\log p)$ for a carefully selected real number $T=T(y)$ to obtain a lower bound on the difference that is of the same size, up to a constant, as the difference in \eqref{repel example}:
\[
\Big| \frac 1{x^{1 + it_i}} \sum_{n\le x} f(n) - \frac 1{(x/y)^{1 + it_1}} \sum_{n\le x/y} f(n) \Big|
\geq c_\epsilon     \Big(\frac{ \log y}{\log x}\Big)^{1-2/\pi}  .
\]
This gives examples, when $\kappa=1$, showing that the exponent $1-\frac 2\pi$   in Theorem \ref{LipThm} cannot be made larger.
   \medskip

\noindent \textsl{Example 3}: \ \underline{$f(p)=g_+(\log p)$ where $g_+(t)=(1+g(t))/2$}. \
Here $g_+(t)=1$ if $-\frac 14\leq t\leq \frac 14$, and $g_+(t)=0$ otherwise, so that $f(n)\geq 0$ for all $n\geq 1$. One can show that 
$$ 
\sum_{n\le x} f(n) =   
\frac{ c_0x } {(\log x)^{1/2}}  +
\left( \text{Re}( c  x^{2i\pi } +o(1) \right)
  x (\log x)^{1/\pi-1}  ;
$$ 
and one can develop this example to show that, when $\kappa=1$,   the exponent $1-1/\pi$  for real, non-negative $f$ in Theorem \ref{LipThm} cannot be made larger.

\appendix


\begin{thebibliography}{10}


\bibitem{BGS}
R. de la Bret\`eche, A. Granville and K. Soundararajan, \emph{Exponential sums with multiplicative coefficients and applications}, preprint.

\bibitem{Dav} 
H. Davenport, \emph{Multiplicative number theory.}  Springer GTM {\bf 74}, (2000).

\bibitem{Ell1}
P.D.T.A. Elliott, \emph{Extrapolating the mean-values of multiplicative functions}, Nederl. Akad. Wetensch. Indag. Math. \textbf{51}, no. 4 (1989), 409--420.

\bibitem{Ell2}
P.D.T.A. Elliott, \emph{Multiplicative functions on arithmetic progressions. VII. Large moduli}, J. London Math. Soc.  \textbf{66} (2002),   14--28.   

\bibitem{Ell3}
P.D.T.A. Elliott, \emph{The least prime primitive root and Linnik's theorem}, Number theory for the millennium, I (Urbana, IL, 2000), A K Peters, Natick, MA, 2002, 393--418. 

\bibitem{FI}
J.B. Friedlander and H. Iwaniec, \emph{Opera de Cribro}, AMS Colloquium Publications vol. 57, 2010.


\bibitem{IV}  A.  Granville, A. J. Harper, and K. Soundararajan, \emph{Mean values of multiplicative functions over function fields.}  Research in 
number theory (2015) 1: 25. doi:10.1007/s40993-015-0023-5

\bibitem{III} 
A. Granville, A. J. Harper, and K. Soundararajan, \emph{A more intuitive proof of a sharp version of Hal{\' a}sz's theorem}. Preprint.  

\bibitem{II}
A. Granville, A. J. Harper, and K. Soundararajan, \emph{Refining Hal{\' a}sz's theorem to obtain asymptotics}.  In progress. 






\bibitem{GSDecay}
A. Granville and K. Soundararajan, \emph{Decay of mean values of multiplicative functions}, Canad. J. Math. \textbf{55}, no. 6 (2003), 1191--1230.

\bibitem{GS}
A. Granville and K. Soundararajan, \emph{The spectrum of multiplicative functions}, Annals of Math.
\textbf{153} (2001), 407--470.

\bibitem{GS2}
A. Granville and K. Soundararajan, \emph{The pretentious approach to analytic number theory}.  Tentative title for 
a book in progress.





\bibitem{Hal1} G. Hal\'asz, \emph{\"Uber die Mittelwerte multiplikativer zahlentheoretischer Funktionen},
Acta Math. Acad. Sci. Hungar. {\bf 19} (1968), 365--403. 


\bibitem{Hal2} G. Hal\'asz, \emph{On the distribution of additive and the mean values of multiplicative arithmetic functions}
Studia Sci. Math. Hungar. 6 (1971), 211--233. 
 


 


\bibitem{Koukoulopoulos}
D. Koukoulopoulos, \emph{Pretentious multiplicative functions and the prime number theorem for arithmetic progressions}, Compos. Math. \textbf{149}, no. 7 (2013), 1129--1149.

\bibitem{Mon} H.L. Montgomery,
\emph{A note on the mean values of multiplicative functions},
Inst. Mittag-Leffler, (Report \# 17)



\bibitem{MV} H.L. Montgomery and R.C. Vaughan, 
\emph{Mean values of multiplicative functions},
Period. Math. Hungar. \textbf{43} (2001), no. 1-2, 199-214. 

\bibitem{Shiu}
P. Shiu, \emph{A Brun-Titchmarsh theorem for multiplicative functions}, J. Reine Angew. Math. \textbf{313} (1980), 161--170. 

%

\bibitem{Sou} 
K. Soundararajan, \emph{Weak subconvexity for central values of $L$-functions}.  Ann. of Math., \textbf{172}, (2010), 1469--1498.

\bibitem{Ten} G. Tenenbaum,
\emph{Introduction to analytic and probabilistic number theory},
Cambridge University Press, 1995.



\end{thebibliography}
\end{document}